\documentclass[12pt]{article}
\usepackage[reqno]{amsmath}
\usepackage{amssymb, enumerate, amsthm}
\usepackage{graphicx}
\usepackage{float}
\usepackage{array}
\usepackage{mathtools}
\usepackage[svgnames]{xcolor}
\usepackage{tikz}
\usepackage{bbm}  
\usepackage{bbold}

\usepackage{fullpage}

\pgfdeclarelayer{edgelayer}
\pgfdeclarelayer{nodelayer}
\pgfsetlayers{edgelayer,nodelayer,main}

\usetikzlibrary{decorations.markings}

\tikzstyle{new}=[circle,  minimum width=4pt,inner sep=0pt, fill=black,draw=black]
\tikzstyle{none}=[circle,fill=white,draw=black]
\tikzstyle{n}=[shape=rectangle,minimum width=1pt,inner sep=0pt, fill=none,draw=none]
\tikzstyle{emph}=[circle,  minimum width=4pt,inner sep=0pt, fill=magenta,draw=magenta]

\tikzset{directed/.style={decoration={
  markings,
  mark=at position .6 with {\arrow{>}}},postaction={decorate}}}

\expandafter\ifx\csname pdfpagebox\endcsname\relax\else
\fi

\newtheorem{theorem}{Theorem}[section]
\newtheorem{lemma}[theorem]{Lemma}
\newtheorem{corollary}[theorem]{Corollary}
\newtheorem{proposition}[theorem]{Proposition}

\DeclarePairedDelimiter\ceil{\lceil}{\rceil}
\DeclarePairedDelimiter\floor{\lfloor}{\rfloor}

\newcommand\arcset{E} 

\newcommand\ka{\kappa}
\newcommand\la{\lambda}

\newcommand\cM{{\mathcal M}}

\newcommand\cW{{\mathcal W}}

\newcommand\Za{{\mathbf a}}

\newcommand\Zv{{\mathbf v}}
\newcommand\Zw{{\mathbf w}}
\newcommand\Zy{{\mathbf y}}
\newcommand\Zz{{\mathbf z}}
\newcommand\Zx{{\mathbf x}}
\newcommand\Ze{{\mathbf e}}

\newcommand\ones{{\mathbbm 1}}
\newcommand\nils{{\mathbf 0}}

\newcommand\cx{{\mathbb C}}

\newcommand\ints{{\mathbb Z}}
\newcommand\re{{\mathbb R}}
\newcommand\rats{{\mathbb Q}}

\newcommand\comp[1]{{\mkern2mu\overline{\mkern-2mu#1}}}

\newcommand\pmat[1]{\begin{pmatrix} #1 \end{pmatrix}}
\newcommand\digr{\vec{\mkern-1.5mu D}\mkern1.5mu}

\DeclareMathOperator\imag{Im}

\DeclareMathOperator\rk{rk}
\DeclareMathOperator\tr{tr}

\DeclareMathOperator\wt{wt}

\DeclareMathOperator\diag{diag}
\DeclareMathOperator\scirc{S}

\DeclareMathOperator\dist{dist}

\newcommand\indeg{d^-}
\newcommand\outdeg{d^+}
\newcommand\innhd[2]{N^-_{#1}(#2)}
\newcommand\outnhd[2]{N^+_{#1}(#2)}

\title{Hermitian adjacency matrix of digraphs and mixed graphs}

\author{Krystal Guo\thanks{Supported in part by NSERC.}\\[1mm]
   {Department of Mathematics}\\{Simon Fraser University}\\{Burnaby, B.C. V5A 1S6} \\ \texttt{krystalg@sfu.ca}
    \and
   Bojan Mohar\thanks{Supported in part by an NSERC Discovery Grant (Canada), by the Canada Research Chair program, and by the Research Grant of ARRS (Slovenia).}~~\thanks{On leave from: IMFM \& FMF, Department of Mathematics, University of Ljubljana, Ljubljana, Slovenia.} \\[1mm]
 {Department of Mathematics}\\{Simon Fraser University}\\{Burnaby, B.C. V5A 1S6}\\ \texttt{mohar@sfu.ca}
 }

\begin{document}
\maketitle

\begin{abstract}
  The paper gives a thorough introduction to spectra of digraphs via its Hermitian adjacency matrix. This matrix is indexed by the vertices of the digraph, and the entry corresponding to an arc from $x$ to $y$ is equal to the complex unity $i$ (and its symmetric entry is $-i$) if the reverse arc $yx$ is not present. We also allow arcs in both directions and unoriented edges, in which case we use $1$ as the entry. This allows to use the definition also for mixed graphs. This matrix has many nice properties; it has real eigenvalues and the interlacing theorem holds for a digraph and its induced subdigraphs. Besides covering the basic properties, we discuss many differences from the properties of eigenvalues of undirected graphs and develop basic theory. The main novel results include the following. Several surprising facts are discovered about the spectral radius; some consequences of the interlacing property are obtained; operations that preserve the spectrum are discussed -- they give rise to an incredible number of cospectral digraphs; for every $0\le\alpha\le\sqrt{3}$, all digraphs whose spectrum is contained in the interval $(-\alpha,\alpha)$ are determined.
\end{abstract}



\section{Introduction}
\label{sec:intro}

Investigation of eigenvalues of graphs has a long history. From early applications in mathematical chemistry (see \cite{CoSi57} or Chapter 8 of \cite{CDS95}), graph eigenvalues found use in combinatorics (see \cite{Bi93, CDS95, GR, BH}), combinatorial optimization (see \cite{Lo79,MoPo93}) and most notably in theoretical computer science (see \cite{Sp12} and references therein), where it became one of the standard tools.

On the other hand, results about eigenvalues of digraphs are sparse. One reason is that it is not clear which matrix associated to a digraph $D$ would best reflect interesting combinatorial properties in its spectrum. One candidate is the \emph{adjacency matrix} $A=A(D)$ whose $(u,v)$-entry is 1 if there is an arc from the vertex $u$ to $v$, and 0 otherwise. A well-known theorem of Wilf \cite{Wilf} bounding the chromatic number in terms of its largest eigenvalue extends to this setting as shown in \cite{M10}. However, the disadvantage of this matrix is that it is not symmetric and we lose the property that eigenvalues are real. Moreover, the algebraic and geometric multiplicities of eigenvalues may be different. Another candidate is the \emph{skew-symmetric adjacency matrix} $S(D)$, where the $(u,v)$-entry is 1 if there is an arc from $u$ to $v$, and $-1$ if there is an arc from $v$ to $u$ (and $0$ otherwise). This choice is quite natural but it only works for \emph{oriented graphs} (i.e.\ when we have no digons). We refer to a survey by Cavers et al.\ \cite{CaCietal12}.

For a purpose in the Ph.D. thesis of one of the authors \cite{Guothesis}, the second author suggested to use the \emph{Hermitian adjacency matrix} $H(D)$ of a digraph instead. This matrix is Hermitian and has many of the properties that are most useful for dealing with undirected graphs. For example, there is the eigenvalue interlacing property for eigenvalues of a digraph and its induced subdigraphs (see Section \ref{sec:interlacing}). In the meantime, Liu and Li independently introduced the same matrix in \cite{LiuLi15} and used it to define and study energy of mixed graphs.

In the Hermitian adjacency matrix, the $(u,v)$-entry is the imaginary unit $i$ if there is an arc from $u$ to $v$, $-i$ if there is an arc from $v$ to $u$, $1$ if both arcs exist and $0$ otherwise. This notion extends to the setting of partially oriented graphs or \emph{mixed graphs} (see Section \ref{sec:defn} for details).

In this paper, we investigate basic properties of the Hermitian adjacency matrix, complementary to the results of Liu and Li \cite{LiuLi15}. We observe that the largest eigenvalue of the Hermitian adjacency matrix is upper-bounded by the maximum degree of the underlying graph and we characterize the case when equality is attained (see Theorem \ref{thm:rhodelta}). This is very similar to what happens with undirected graphs. On the other hand, several properties that follow from Perron-Frobenius theory are lost, and the differences that occur are discussed. For example, it may happen that the spectral radius $\rho(X)$ of a digraph $X$ is not equal to the largest eigenvalue $\lambda_1(X)$ but is equal to the absolute value of its smallest eigenvalue. To some surprise, things cannot go arbitrarily ``bad'', as we are able to prove that
$$
   \lambda_1(X) \le \rho(X)\le 3\lambda_1(X).
$$
Both of these inequalities are tight. See Theorem \ref{thm:max negative rho} and examples preceding it.

Similarly, in another contrast to the case of the adjacency matrix, there does not appear to be a bound on the diameter of  the digraph in terms of the number of distinct eigenvalues of the Hermitian adjacency matrix. In fact, there is an infinite family of strongly connected digraphs whose number of distinct eigenvalues is constant, but whose diameter goes to infinity.

Despite the many unintuitive properties that the Hermitian adjacency matrix exhibits, it is still possible to extract combinatorial structure of the digraph from its eigenvalues. In Section \ref{sec:hevalspm1}, we find all digraphs whose $H$-eigenvalues lie in the range $(-\alpha,\alpha)$ for any $0\le\alpha\le\sqrt{3}$. Using interlacing, one can find spectral bounds for the maximum independent set and maximum acyclic subgraphs for oriented graphs (see Section \ref{sec:interlacing}).

In Section \ref{sec:cospec} we discuss cospectral digraphs. As expected, cospectral pairs are not rare. Moreover, there are several kinds of ``switching'' operations that change the digraph but preserve its spectrum. In particular, the set of all digraphs obtained by orienting the edges of an arbitrary undirected graph of order $n$ usually contains up to $2^n$ non-isomorphic digraphs cospectral to each other. This kind of questions occupies Section \ref{sec:cospec}.

\section{Definitions}
\label{sec:defn}

A \emph{directed graph} (or \emph{digraph}) $X$ consists of a finite set $V(X)$ of vertices together with a subset $E(X) \subseteq V(X)\times V(X)$ of ordered pairs called \emph{arcs} or \emph{directed edges}. Instead of $(x,y) \in E(X)$, we write $xy$ for short. If $xy \in E$ and $yx \in E$, we say
that the unordered pair $\{x,y\}$ is a \emph{digon} of $X$.

A mixed graph is a graph where both directed and undirected edges may exist. More formally, a \emph{mixed graph} is an ordered triple $(V, E, A)$ where $V$ is the vertex-set, $E$ is a set of undirected edges, and $A$ is set of arcs, or directed edges. In this paper, the Hermitian adjacency matrix is defined in such a way that the undirected edges may be thought of as digons and, from this perspective, mixed graphs are equivalent to the class of digraphs that we consider here.

If a digraph $X$ has every arc contained in a digon, then we  say that $X$ is \emph{undirected} or, more simply, that $X$ is a \emph{graph.}
The \emph{underlying graph of a digraph $X$,} denoted $\Gamma(X)$, is the graph with vertex-set $V(X)$ and edge-set \[E = \{ \{x,y\} \mid xy \in \arcset(X) \text{ or } yx \in \arcset(X) \}.\] If a digraph $X$ has no digons, we say that $X$ is an \emph{oriented graph.} Given a graph $G$, the \emph{digraph} of $G$ is the digraph on the same vertex set with every undirected edge replaced by a digon. We will denote the digraph of a graph $G$ by $\digr(G)$.

The \emph{converse} of a digraph $X$ is the digraph $X^C$ with the same vertex set and arc set
$\arcset\left(X^C\right) = \{xy \mid yx \in \arcset \}$.
Observe that every digon of $X$ is unchanged under the operation of taking the converse.

For a vertex $x \in V(X)$, we define the set of \emph{in-neighbours} of $x$ as $\innhd{X}{x} = \{u \in V(X) \mid ux \in \arcset(X)\}$ and the set of \emph{out-neighbours} as $\outnhd{X}{x} = \{u \in V(X) \mid xu \in \arcset(X)\}$.
The \emph{in-degree} of $x$, denoted $\indeg(x)$, is the number of  in-neighbours of $x$. The \emph{out-degree} of  $x$, denoted $\outdeg(x)$, is the number of out-neighbours of $x$.
The maximum in-degree (resp. out-degree) of $X$ will be denoted $\Delta^-(X)$ (resp. $\Delta^+(X)$).

For a digraph $X$ with vertex set $V= V(X)$ and arc set $\arcset = \arcset(X)$, we consider the \emph{Hermitian adjacency matrix} $H = H(X) \in \cx^{V\times V}$, whose entries $H_{uv} = H(u,v)$ are given by
\[
H_{uv} = \begin{cases} ~1 & \text{if } uv \text{ and } vu \in \arcset ; \\
~i & \text{if } uv \in \arcset \text{ and } vu \notin \arcset ; \\
-i & \text{if } uv \notin \arcset \text{ and } vu \in \arcset ; \\
~0 & \text{otherwise.}\end{cases}
\]
If every edge of $X$ lies in a digon, then $H(X) = A(X)$, which reflects that $X$ is, essentially, equivalent to an undirected graph. This definition also covers mixed graphs, for which we replace each undirected edge by a digon. 

In general, we define the Hermitian adjacency matrix of a mixed multigraph $X = (V,E,A)$ without digons to be the matrix $H = H(X)$ with entries given by $H_{uv} = a + bi - c i$, where $a$ is the multiplicity of $uv$ as an undirected edge, $b$ is the multiplicity of $uv$ as an arc in $A$ and $c$ is the multiplicity of $vu$ as an arc of $A$. (Note that either $b=0$ or $c=0$ since there are no digons.) In this paper, we will mainly restrict our attention to \emph{simple} digraphs and our results will hold for simple mixed graphs when they are considered as simple digraphs obtained by replacing each undirected edge by a digon. Having said this, we will from now on speak only of digraphs.

Observe that $H$ is a Hermitian matrix and so is diagonalizable with real eigenvalues. The following proposition contains properties that are true for adjacency matrices which also carry over to the Hermitian case.

\begin{proposition}\label{prop:real}
For a digraph $X$ on $n$ vertices and $H = H(X)$ its Hermitian adjacency matrix, the following are true:
\begin{enumerate}[\rm (i)]
\item All eigenvalues of $H$ are real numbers.
\item The matrix $H$ has $n$ pairwise orthogonal eigenvectors in $\cx^n$ and so $H$ is unitarily similar to a diagonal matrix.
\end{enumerate}
\end{proposition}

The eigenvalues of $H(X)$ are the \emph{$H$-eigenvalues of $X$} and the spectrum of $H(X)$ (i.e.\ the multiset of eigenvalues, counting their multiplicities) is the \emph{$H$-spectrum of $X$}. We will denote the $H$-spectrum by $\sigma_H(X)$ and we will express it as either a multiset of $H$-eigenvalues or a list of distinct $H$-eigenvalues with multiplicities in superscripts.

The $H$-eigenvalues of a digraph $X$ will be ordered in the decreasing order, the $j$th largest eigenvalue will be denoted by $\la_j(X) = \la_j(H)$, so that $\la_1(X)\ge \la_2(X)\ge \cdots \ge \la_n(X)$. The characteristic polynomial of $H(X)$ will be denoted by $\phi(H(X),t)$ or simply by $\phi(X,t)$.

A direct consequence of Proposition \ref{prop:real} is the min-max formula for $\la_j(X)$:
\begin{equation}
  \la_j(X) ~= \max_{\dim U=j} \min_{\Zz\in U} \Zz^*H\Zz ~= \min_{\dim U=n-j+1} \max_{\Zz\in U} \Zz^*H\Zz
\label{eq:minmax jth largest}
\end{equation}
where $H=H(X)$ and the outer maximum (minimum) is taken over all subspaces $U$ of dimension $j$ ($n-j+1$) and the inner minimum (maximum) is taken over all $\Zz$ of norm $1$.

We say that digraphs $X$ and $Y$ are \emph{$H$-cospectral} if matrices $H(X)$ and $H(Y)$ are cospectral, i.e.\ they have the same characteristic polynomials, $\phi(X,t) = \phi(Y,t)$.
Recall that digraphs $X$ and $Y$ are cospectral (or \emph{$A$-cospectral}, if we wish to distinguish between the two matrices) if $A(X)$ and $A(Y)$ have the same characteristic polynomial. To avoid ambiguity, we refer to eigenvalues and spectrum of $X$ with respect to its adjacency matrix as the \emph{$A$-eigenvalues} and \emph{$A$-spectrum,} respectively.

\section{Basic properties}\label{sec:basics}

We first examine the coefficients of the characteristic polynomial of $H(X)$.

\begin{lemma}
\label{lem:Hdegs}
For $X$ a digraph and $H = H(X)$ its Hermitian adjacency matrix,
$$
(H^2)_{uu} = d(u)
$$
where $d(u)$ is the degree of $u$ in the underlying graph of $X$.
\end{lemma}

\begin{proof}
Since $H$ is Hermitian and has only entries $0$, $1$ and $\pm i$, we have
\[
H_{uv} H_{vu} = H_{uv} \comp{H_{uv}} = 1
\]
whenever $H_{uv} \neq 0$. This implies that the $(u,u)$ diagonal entry in $H^2$ is the degree of $u$ in the underlying graph of $X$.
\end{proof}

Note that the degree of a vertex $x$ of digraph $X$ in the underlying graph of $X$ is equal to $|\innhd{X}{x} \cup \outnhd{X}{x}|$. Lemma \ref{lem:Hdegs} gives the following information about the coefficient of $t^{n-2}$ in the characteristic polynomial of a digraph on $n$ vertices.

\begin{corollary}
\label{cor:cn-2}
Let $X$ be a digraph on $n$ vertices. The coefficient of $t^{n-2}$ in the characteristic polynomial $\phi(H(X),t)$, is $-e$ where $e$ is the number of edges of the underlying graph of $X$.
\end{corollary}

\begin{proof} Let $\Gamma$ be the underlying graph of $X$ and let $H = H(X)$. If $\lambda_1, \ldots, \lambda_n$ are the eigenvalues of $H$, the characteristic polynomial of $H$ can be written as
\[
\phi(H,t) = (t- \lambda_1) \cdots (t-\lambda_n).
\]
Thus, the coefficient of $t^{n-2}$ is
\[
c_2 = \sum_{1 \leq j < k \leq n} \lambda_j \lambda_k.
\]
Observe that
\[
\left(\sum_{j = 1}^{n} \lambda_j\right)^2 =  \sum_{j =1}^n \lambda_j^2 + 2\sum_{1 \leq j < k \leq n} \lambda_j \lambda_k = \tr(H^2) + 2c_2.
\]
The matrix $H$ has all zeros on the diagonal, and so $\sum_{j = 1}^{n} \lambda_j = 0$. Then, we obtain that
\[
   \tr(H^2) + 2c_2 = 0.
\]
By Lemma \ref{lem:Hdegs}, we have that $\tr(H^2) = \sum_{u \in V(X)} d_{\Gamma}(u) = 2e$, where $e$ is the number of edges of $\Gamma$. Thus $c_2 = -e$.
\end{proof}

Using the definition of matrix determinant as the sum of contributions over all permutations, we can write the characteristic polynomial of $H(X)$ in terms of cycles in the underlying graph as follows.
Let $X$ be a digraph on $n$ vertices. A \emph{basic subgraph of order $j$} of $X$ is a subgraph of the underlying graph $\Gamma(X)$ with $j$ vertices, each of whose components is either a single edge or a cycle in $\Gamma(X)$, and when it is a cycle, the corresponding digraph in $X$ has an even number of directed (non-digon) edges. For each basic subgraph $B$, let $c(B)$ denote the number of cycles in $B$, and let $r(B)$ be $\frac{1}{2}|f-b|$, where $f$ is the number of forward arcs and $b$ is the number of backward arcs in $B$ (for some orientation of the cycles in $B$; we are only interested in the value of $r(B)$ modulo 2, which is independent of the chosen orientations). Finally, let $s(B)$ be the number of components of $B$ with even number of vertices (i.e. edges and even cycles).

The following result, which corresponds to a well-known result of Sachs \cite{Sachs64} (see also \cite[Section 1.4]{CDS95}), appears in \cite[Theorem 2.8]{LiuLi15}.

\begin{theorem}[\cite{LiuLi15}]
\label{thm:hcp}
Let $X$ be a digraph of order $n$. Then the characteristic polynomial $\phi(X,t) = \sum_{j = 0}^{n} c_j t^j$ has coefficients equal to
$$
   c_j = \sum_{B} (-1)^{r(B)+s(B)} 2^{c(B)},
$$
where the sum runs over all basic subgraphs $B$ of order $n-j$ in $\Gamma(X)$.
\end{theorem}

In particular, the expression for $c_{n-2}$ from Theorem \ref{thm:hcp} yields the result in Corollary \ref{cor:cn-2} since all basic graphs of order 2 are single edges. Using this theorem, we also obtain the following corollary, which has also been used in \cite{LiuLi15}.

\begin{corollary}[\cite{LiuLi15}]
\label{cor:cutedge}
Suppose that $\{u,v\}$ is a digon in $X$. If $uv$ is a cut-edge of\/ $\Gamma(X)$, then the spectrum of $H(X)$ is unchanged when the digon is replaced with a single arc $uv$ or $vu$.
\end{corollary}

Analogous to the results for the adjacency matrix found in standard texts like \cite{Bi93}, we may write the $(u,v)$-entry of $H(X)^k$ as a weighted sum of the walks in $\Gamma(X)$ of length $k$ from $u$ and $v$.

\begin{proposition}
\label{prop:hpath}
Let $X$ be a digraph and $H=H(X)$. For vertices $u,v \in V(X)$ and any positive integer $k$, the $(u,v)$-entry of the $k$th power of $H$ can be expressed is as follows:
\[
\left(H^k\right)_{uv} = \sum_{W \in \cW} \wt(W)
\]
where $\cW$ is the set of all walks of length $k$ from $u$ to $v$ in $\Gamma(X)$ and for $W = (v_0, \ldots, v_k) \in \cW$, where $v_0 = u$ and $v_k = v$, the \emph{weight} is
\[
\wt(W) = \prod_{j = 0}^{k-1} H(v_j,v_{j+1}).
\]
\end{proposition}

We will use Proposition \ref{prop:hpath} to find an expression for $\tr(H(X)^3)$ in terms of numbers of sub-digraphs isomorphic to certain types of triangles.

\begin{proposition}
\label{prop:traceh3}
Let\/ $X$ be a digraph of order $n$ and $\lambda_1,\dots,\lambda_n$ be its $H$-eigenvalues. Then:
\begin{enumerate}[\rm (i)]
\item
$\sum_{j=1}^n \lambda_j = 0$.
\item
$\sum_{j=1}^n \lambda_j^2 = 2e$, where $e$ is the number of edges of $\Gamma(X)$.
\item
$\sum_{j=1}^n \lambda_j^3 = 6(x_2 + x_3 + x_4 - x_1)$,
where $x_j$ is the number of copies of $X_j$ as an induced sub-digraph of $X$ and $X_j$ $(1\le j\le 4)$ are the digraphs shown in Figure~\ref{fig:c3str3}.
\end{enumerate}
\end{proposition}

\begin{proof}
The first two statements are easy consequences of the facts that $\tr(H) = 0$ and $\tr(H^2) = 2e$.
Below we provide details for (iii).

Let $X$ be a digraph with $V = V(X)$ and $\arcset = \arcset(X)$. We apply Proposition \ref{prop:hpath} to obtain:
\[
\tr(H^3) = \sum_{v \in V} (H^3)_{v,v} = \sum_{v\in V} \sum_{W \in \cW_v} \wt(W)
\]
where $W_v$ is the set of closed walks at $v$ of length $3$ in $\Gamma(X)$ and $\wt$ as defined in Proposition \ref{prop:hpath}. Every closed walk of length $3$ has $C_3$ as its underlying graph. For simplicity, we will refer to any digraph with $C_3$ as the underlying graph as a \emph{triangle} and we will call a sub-digraph of $X$ isomorphic to a triangle a \emph{triangle of $X$.} In Figure \ref{fig:c3str3}, we list all non-isomorphic triangles with an even number of arcs (non-digon edges). Those with an odd number make contributions to $\tr(H)$ that cancel each other out. Finally, note that the weight of $X_1$ is $-1$, while the weight of the others is 1.
\end{proof}

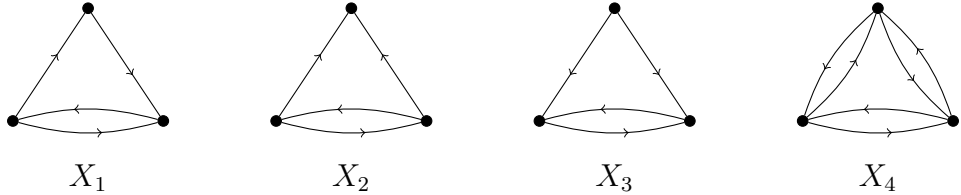
\begin{figure}[ht]
\centering
\begin{tikzpicture}
	\begin{pgfonlayer}{nodelayer}
		\node [style=new] (0) at (-5, -1.75) {};
		\node [style=new] (1) at (-3, -1.75) {};
		\node [style=new] (2) at (-4, -0.25) {};
		\node [style=new] (3) at (-1.5, -1.75) {};
		\node [style=new] (4) at (0.5, -1.75) {};
		\node [style=new] (5) at (-0.5, -0.25) {};
		\node [style=new] (6) at (2, -1.75) {};
		\node [style=new] (7) at (4, -1.75) {};
		\node [style=new] (8) at (3, -0.25) {};
		\node [style=new] (9) at (5.5, -1.75) {};
		\node [style=new] (10) at (7.5, -1.75) {};
		\node [style=new] (11) at (6.5, -0.25) {};
		\node [style=n] (12) at (-4, -2.5) {$X_1$};
		\node [style=n] (13) at (-0.5, -2.5) {$X_2$};
		\node [style=n] (14) at (3, -2.5) {$X_3$};
		\node [style=n] (15) at (6.5, -2.5) {$X_4$};
	\end{pgfonlayer}
	\begin{pgfonlayer}{edgelayer}
		\draw [directed] (2) to (1);
		\draw [directed, bend right=15] (1) to (0);
		\draw [directed] (0) to (2);
		\draw [directed, bend right=15] (4) to (3);
		\draw [directed] (3) to (5);
		\draw [directed] (8) to (7);
		\draw [directed, bend right=15] (11) to (10);
		\draw [directed, bend right=15] (10) to (9);
		\draw [directed, bend right=15] (9) to (11);
		\draw [directed, bend right=15] (0) to (1);
		\draw [directed, bend right=15] (3) to (4);
		\draw [directed] (4) to (5);
		\draw [directed, bend right=15] (7) to (6);
		\draw [directed] (8) to (6);
		\draw [directed, bend right=15] (6) to (7);
		\draw [directed, bend right=15] (11) to (9);
		\draw [directed, bend right=15] (9) to (10);
		\draw [directed, bend right=15] (10) to (11);
	\end{pgfonlayer}
\end{tikzpicture}
\caption{All non-isomorphic digraphs with even number of arcs whose underlying graph is the $3$-cycle.\label{fig:c3str3}}
\end{figure}

We define $G(X)$, the \emph{symmetric subgraph of a digraph $X$,} to be the graph with vertex set $V(X)$ and the edge set being the set consisting of all digons of $X$. Similarly, we define $D(X)$, the \emph{asymmetric sub-digraph of $X$}, to be the digraph with vertex set $V(X)$ and the arc set being the set of arcs of $X$ which are not contained in any digons of $X$.

\begin{lemma}
Let $X$ be a digraph. Then $H(X)$ has the all-ones vector $\ones$ as an eigenvector if and only if $G(X)$ is regular and $D(X)$ is eulerian.
\end{lemma}

\begin{proof}
Suppose $D(X)$ is eulerian, then the sum along any row of $H(D(X))$ is equal to 0 and so $\ones$ is an eigenvector for $H(D(X))$ of eigenvalue 0. If $G(X)$ is regular, then $\ones$ is an eigenvector of $H(G(X)) = A(G(X))$ with eigenvalue equal to the valency of $G(X)$, say $k$. Then
\[
H(X) \ones = H(D(X))\ones + H(G(X)) \ones = k \ones.
\]

For the other direction, suppose that $H(X) \ones = \gamma \ones$ for some $\gamma \in \re$.  The row sum along the row indexed by $x$ is equal  to $d + (s - t)i$ where $d = d(x)$, $s= \indeg(x)$, and $t = \outdeg(x)$. Then, since $\gamma$ is a real number, we must have that $s =t$ and $d = \gamma$.
\end{proof}

\section{Interlacing}
\label{sec:interlacing}

Suppose that $\la_1\ge \la_2\ge \cdots \ge \la_n$ and $\ka_1\ge \ka_2\ge \cdots \ge \ka_{n-t}$ (where $t\ge1$ is an integer) be two sequences of real numbers. We say that the sequences $\la_l$ ($1\le l\le n$) and $\ka_j$ ($1\le j\le n-t$) \emph{interlace} if for every $s=1,\dots,n-t$, we have
$$
    \la_s\ge \ka_s\ge \la_{s+t}.
$$

The usual version of the eigenvalue interlacing property states that the eigenvalues of any principal submatrix of a Hermitian matrix interlace those of the whole matrix (see \cite[Theorems 4.3.8 and 4.3.15]{HoJo13}).

\begin{theorem}
\label{thm:interlacingGR}
If $H$ is a Hermitian matrix and $B$ is a principal submatrix of $H$, then the eigenvalues of $B$ interlace those of $H$.
\end{theorem}

Interlacing of eigenvalues is a powerful tool in algebraic graph theory.
Theorem \ref{thm:interlacingGR} implies that the eigenvalues of any induced subdigraph interlace those of the digraph itself.

\begin{corollary}
\label{cor:interlacing digraph}
The eigenvalues of an induced subdigraph interlace the eigenvalues of the digraph.
\end{corollary}

To see a simple example how useful the interlacing theorem is, let us consider the following notion.
Let $\eta^+(X)$ denote the number of non-negative $H$-eigenvalues of a digraph $X$ and $\eta^{-}(X)$ denote the number of non-positive $H$-eigenvalues of $X$.

\begin{theorem}
If a digraph $X$ contains a subset of $m$ vertices, no two of which form a digon, then
$\eta^+(X) \geq \ceil*{ \frac{m}{2} }$
and $\eta^{-}(X) \geq \ceil*{\frac{m}{2}}$.
\end{theorem}

\begin{proof}
Let $\lambda_1 \geq \cdots \geq \lambda_n$ be the $H$-eigenvalues of $X$. Let $U\subseteq V(X)$, $|U|=m$, be a subset without digons, and let $B$ be the principal submatrix of $H(X)$ with rows and columns corresponding to $U$. As shown in \cite{LiuLi15} (see also Theorem \ref{thm:h-ori} in this paper), the $H$-eigenvalues $\mu_j$ ($1\le j\le m$) of $B$ are symmetric about 0. Therefore, $\mu_{\ceil*{\frac{m}{2}}} \ge 0$ and $\mu_{\floor*{\frac{m}{2}}+1} \le 0$.
By interlacing, we see that $\lambda_{\ceil*{ \frac{m}{2} }} \ge \mu_{\ceil*{\frac{m}{2}}} \ge  0$ and $\la_{n - \ceil*{\frac{m}{2}} + 1} \le \mu_{\floor*{\frac{m}{2}}+1} \le 0$. This implies the result.
\end{proof}

Similarly, the Cvetkovi\'{c} bound (see \cite{CDS95}) for the largest independent set of a graph extends to digraphs and their Hermitian adjacency matrix.

\begin{proposition}
If $X$ has an independent set of size $\alpha$, then $ \eta^+(X) \geq \alpha$
and $\eta^-(X) \geq \alpha$.
\end{proposition}

\begin{proof} Let $\lambda_1 \geq \cdots \geq \lambda_n$ be the eigenvalues of $H(X)$. By interlacing, we see that $\lambda_{\alpha} \geq 0$ and so $H(X)$ has at least $\alpha$ non-negative eigenvalues. Applying the same argument to $-H(X)$ shows that there are at least $\alpha$ non-positive eigenvalues as well.
\end{proof}

We can also obtain a spectral bound on the maximum induced transitive tournament of a digraph. In \cite{GrKiSh93}, Gregory, Kirkland and Shader found the tight upper bound on the spectral radius of a skew-symmetric matrix and classified the matrices which attain the bound. We will state a special case of their theorem restricted to the Hermitian adjacency matrices of oriented graphs. Let us recall that two tournaments are \emph{switching-equivalent} if one can be obtained from the other by reversing all arcs across an edge-cut of the underlying graph; see also Section \ref{sec:hfam-transtmt}.

\begin{theorem}[\cite{GrKiSh93}]
If $X$ is an oriented graph of order $n$, then
\[
\lambda_1(H(X)) \leq \cot \left(\frac{\pi}{2n} \right) .
\]
Equality holds if and only if $X$ is switching-equivalent to $T_n$, the transitive tournament of order $n$.
\end{theorem}

The following lemma is another immediate consequence of interlacing.

\begin{corollary}
If $X$ is a digraph with an induced subdigraph that is switching equivalent to $T_m$, then $\lambda_1(H(X)) \geq  \cot \left(\frac{\pi}{2m} \right)$.
\end{corollary}

In other words, if $m > \frac{\pi}{2\cot^{-1}(\lambda_1(H(X)))}$ then $X$ cannot contain an induced subdigraph that would be switching-equivalent to $T_m$.

A more general version of the interlacing theorem (see, e.g. \cite{BH} or \cite{H95}) involves more general orthogonal projections.

\begin{theorem}
\label{thm:interlacingGR_general}
If $H$ is a Hermitian matrix of order $n$ and $S$ is a matrix of order $k\times n$ such that $SS^*=I_k$, then the eigenvalues of $H$ and the eigenvalues of $SHS^*$ interlace.
\end{theorem}

For a digraph $X$, let $\Pi = V_1\cup \ldots \cup V_m$ be a partition of the vertex set of $X$ and order the vertices of $X$ such that $\Pi$ induces the following partition of $H(X)$ into block matrices:
\[
H(X) = \pmat{H_{11} & \ldots & H_{1m} \\
\vdots & \ddots & \vdots \\
H_{m1} & \ldots & H_{mm}}.
\]
The \emph{quotient} matrix of $H(X)$ with respect to $\Pi$ is the matrix $B = [b_{jk}]$ where $b_{jk} $ is the average row sum of the block matrix $H_{jk}$, for $j,k \in [m]$. The following corollary of Theorem \ref{thm:interlacingGR_general} is proved in the same way as the corresponding result for undirected graphs found in \cite{H95}.

\begin{theorem}
\label{thm:interlacing quotient}
Eigenvalues of the quotient matrix of $H(X)$ with respect to a partition of vertices interlace those of $H(X)$.
\end{theorem}

The partition $\Pi$ is said to be \emph{equitable} if each block $H_{jk}$ has constant row sums, for $j,k \in [m]$. In this case, more can be said.

\begin{corollary}
\label{cor:interlacing-equitable}
Let $\Pi$ be an equitable partition of the vertices of $X$ and $B$ is the quotient matrix of $H(X)$ with respect to $\Pi$. If $\lambda$ is an eigenvalue of $B$ with multiplicity $\mu$, then $\lambda$ is an eigenvalue of $H(X)$ with multiplicity at least $\mu$.
\end{corollary}

These results remain true in the setting of digraphs with multiple edges.

\section{Spectral radius}
\label{sec:hperron}
\label{sec:minusrho}

For a digraph $X$, let the eigenvalues of $H =H(X)$ be $\lambda_1 \geq \cdots \geq \lambda_n$. Note that since $H$ is not a matrix with non-negative entries, there is no analogue of the Perron value of the adjacency matrix and the properties of $\lambda_1$ may be highly unintuitive. Figure \ref{fig:k3k4prime} shows a strongly connected digraph $K_3'$ on $3$ vertices with $H$-eigenvalues $\{1^{(2)}, -2\}$. This shows that, in general, $\lambda_1$ is not necessarily simple or largest in magnitude.

Instead of considering the largest eigenvalue, we may consider the largest eigenvalue in absolute value, for which we find a bound that is analogous to the adjacency matrix. The \emph{spectral radius} $\rho(M)$ of a matrix $M$ is defined as
\[ \rho(M) = \max\{|\lambda| \mid \lambda \text{ an eigenvalue of }M \}
\]
and we also define $\rho(X) = \rho(H(X))$ as the \emph{spectral radius} of the digraph $X$.

\begin{theorem}
\label{thm:rhodelta}
If $X$ is a digraph (multiple edges allowed), then $ \rho(X) \leq \Delta(\Gamma(X))$. When $X$ is weakly connected, the equality holds if and only if $\Gamma(X)$ is a $\Delta(\Gamma(X))$-regular graph and there exists a partition of $V(X)$ into four (possibly empty) parts $V_{1}$, $V_{-1}$, $V_{i}$, and $V_{-i}$ such that one of the following holds:
\begin{enumerate}[\rm (i)]
\item For $j \in \{\pm 1, \pm i\}$, the digraph induced by $V_j$ in $X$ contains only digons. Every other arc $uv$ of $X$ is such that $u \in V_j$ and $v \in V_{(-i)j}$ for some $j \in \{\pm 1, \pm i\}$. See Figure \ref{fig:rhodelta}.
\item For $j \in \{\pm 1, \pm i\}$, the digraph induced by $V_j$ in $X$ is an independent set. For each $j \in \{\pm 1, \pm i\}$, every arc with one end in $V_{j}$ and one end in $V_{-j}$ is contained in a digon. Every other arc $uv$ of $X$ is such that $u \in V_j$ and $v \in V_{ij}$ for some $j \in \{\pm 1, \pm i\}$. See Figure \ref{fig:rhodelta}.
\end{enumerate}
\end{theorem}

\begin{figure}
\centering
\includegraphics[scale=0.7]{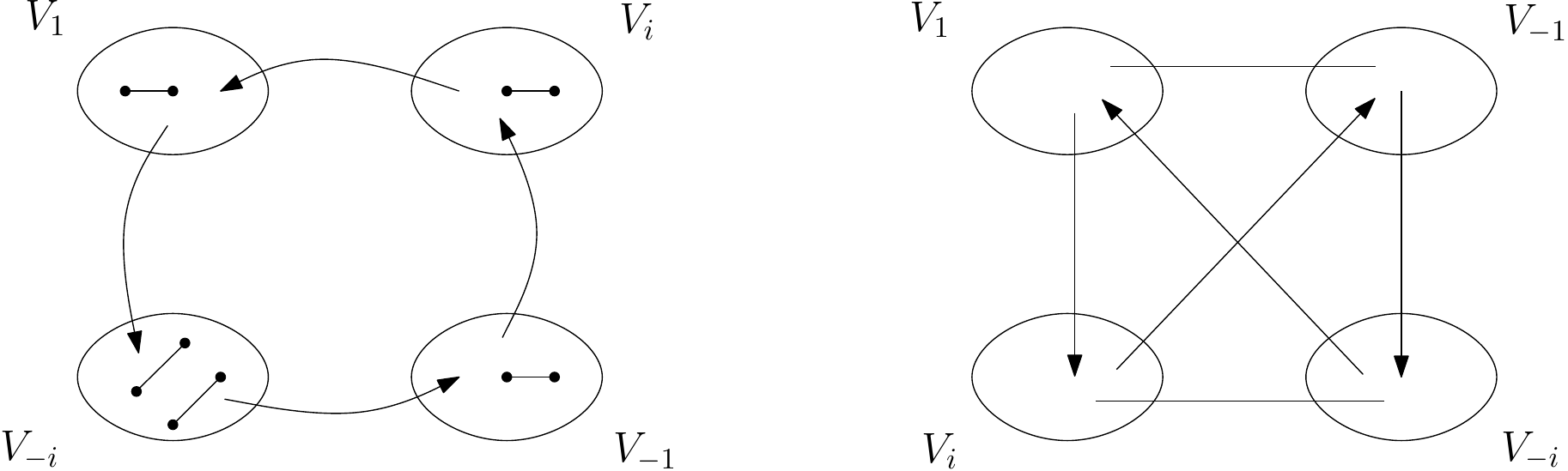}
\caption{Cases (i) and (ii) of Theorem \ref{thm:rhodelta}.}\label{fig:rhodelta}
\end{figure}

\begin{proof}
Let $H =H(X)$ and let $\lambda$ be an eigenvalue of $H$ with eigenvector $\Zx$. Choose $v\in V(X)$ such that $|\Zx(v)|$ is maximum. Now we consider the $v$-entry of $H\Zx$. For simplicity of notation, we will write $N(v) := \innhd{X}{v} \cap \outnhd{X}{v}$. We obtain
\[
(H\Zx)(v) = \sum_{u \in N(v)} \Zx(u) ~+~ i \sum_{w \in \outnhd{X}{v}\setminus N(v)} \Zx(w) ~-~ i \sum_{y \in \innhd{X}{v}\setminus N(v)} \Zx(y).
\]
On the other hand, $(H\Zx)(v) = \lambda \Zx(v)$. Then
\begin{eqnarray}
|\lambda \Zx(v)| &=& |(H\Zx)(v)| \nonumber \\[1mm]
&\leq& \sum_{u \in N(v)} |\Zx(u)|
~+ \sum_{w \in \outnhd{X}{v}\setminus N(v)} |\Zx(w)|
~+ \sum_{y \in \innhd{X}{v}\setminus N(v)} |\Zx(y)| \nonumber \\
&\leq& \sum_{u \in N(v)} |\Zx(v)| ~+ \sum_{w \in \outnhd{X}{v}\setminus N(v)} |\Zx(v)| ~+ \sum_{y \in \innhd{X}{v}\setminus N(v)} |\Zx(v)| \nonumber \\[2mm]
&=& \deg_{\Gamma(X)}(v)\, |\Zx(v)| \label{eq:largeevdeg} \\[2mm]
&\leq& \Delta(\Gamma(X))\, |\Zx(v)|. \nonumber
\end{eqnarray}
From this, we obtain that $|\lambda| \leq \Delta(\Gamma(X))$.

If $ \rho(H(X)) = \Delta(\Gamma(X))$, then all of the inequalities in (\ref{eq:largeevdeg}) must hold with equality. From the last inequality of (\ref{eq:largeevdeg}), we see that $v$ has degree $\Delta(\Gamma(X))$  in $\Gamma(X)$. If the second inequality of (\ref{eq:largeevdeg}) holds with equality, we obtain that $|\Zx(z)| = |\Zx(v)|$ for all $z \in \innhd{X}{v} \cup \outnhd{X}{v}$. Since the choice of $v$ was arbitrary amongst all vertices attaining the maximum absolute value in $\Zx$, we may apply this same argument to any vertex adjacent to $v$ in $\Gamma(X)$. Since $X$ is weakly connected, we obtain that $|\Zx(z)| = |\Zx(v)|$ for any $z \in V(X)$.

The first inequality of (\ref{eq:largeevdeg}) follows from the triangle inequality for sums of complex numbers, and so equality holds if and only if every complex number in $Z$ has the same argument as $\lambda \Zx(v)$, where
\[
\begin{split}
Z ~=~ &\{ \Zx(u) \mid u \in N(v) \}\cup  \\
&\{i\Zx(w) \mid w \in \outnhd{X}{v}\setminus N(v) \}\cup  \\
&\{ -i\Zx(y) \mid y \in \innhd{X}{v}\setminus N(v)\} .
\end{split} \]
We may normalize $\Zx$ such that $\Zx(v) = 1$. There are three cases: $\lambda =0$, $\lambda$ is positive or $\lambda$ is negative. Since we are bounding the spectral radius, we need not consider the $\lambda =0$ case; the only digraph with $\rho(X) = 0$ is the empty graph and the statement of the theorem holds there.

Suppose that $\lambda$ is positive. Every complex number in $Z$ has the same norm and same argument as $\Zx(v)=1$ and is thus equal to $1$. We conclude that
\[
\Zx(z) =
\begin{cases}
1, & \text{if } z \in N(v); \\
-i, & \text{if } z\in \outnhd{X}{v}\setminus N(v); \text{ and } \\
i, & \text{if } z\in \innhd{X}{v}\setminus N(v).
\end{cases}
\]
Repeating the argument at a vertex $w$ such  that $\Zx(w) = -i$, we see that:
\[
\Zx(z) =
\begin{cases}
-i, & \text{if } z \in N(w); \\
-1, & \text{if } z\in \outnhd{X}{w}\setminus N(w); \text{ and } \\
+1, & \text{if } z\in \innhd{X}{w}\setminus N(w).
\end{cases}
\]
Similar argumant can be used when $\Zx(z)=-1$ or $\Zx(z)=i$. From this we conclude that $V(X)$ has a partition into sets $V_{1}$, $V_{-1}$, $V_{i}$, and $V_{-i}$ such that condition (i) of the theorem holds.

Suppose now that $\lambda$ is negative. Every complex number in $Z$ has the same norm and same argument as $-1$ and is hence equal to $-1$. Thus we obtain that
\[
\Zx(z) =
\begin{cases}
-1, & \text{if } z \in N(v); \\
i, & \text{if } z\in \outnhd{X}{v}\setminus N(v); \text{ and } \\
-i, & \text{if } z\in \innhd{X}{v}\setminus N(v).
\end{cases}
\]
Repeating the argument at vertices $w$ such  that $\Zx(w) = -1$ or $\pm i$, we see that $V(X)$ has a partition into $V_{1}$, $V_{-1}$, $V_{i}$, and $V_{-i}$ such that condition (ii) of the theorem holds.

We now consider the converse for the two cases of the theorem. Let $X$ be a digraph such that $\Gamma(X)$ is $k$-regular. Suppose that $V(X)$ has a partition  $V_{1}$, $V_{-1}$, $V_{i}$, and $V_{-i}$ such that condition (i) or (ii) holds. Let $\Zx$ be the vector indexed by the vertices of $X$ such that $\Zx(z) = j$ if $z \in V_j$. Then it is easy to see that for every vertex $v$ we have $(H\Zx)(v) = k\Zx$ (in case (i)) or $(H\Zx)(v) = -k\Zx$ (in case (ii)). Thus $\Zx$ is an eigenvector for $H$ with eigenvalue $\pm k$, and the bound is tight as claimed.
\end{proof}

For undirected graphs, $\rho(X)$ is always larger or equal to the average degree. However, for digraphs, $\rho(X)$ can be smaller than the minimum degree in $\Gamma(X)$. An example is the digraph $\widetilde{C_3}$ shown in Figure \ref{fig:hsymnotbip} which has eigenvalues $\pm \sqrt{3}$ and $0$, while its underlying graph has minimum degree~2. Of course, this anomaly is also justified by Theorem \ref{thm:rhodelta} since $\widetilde{C_3}$ does not have the structure as in Figure \ref{fig:rhodelta}.

Next we shall discuss digraphs whose $H$-spectral radius $\rho(X)$ is larger than the largest $H$-eigenvalue $\la_1(X)$.
In that case, $\rho(X)$ is attained by the absolutely largest negative eigenvalue. First we treat an extremal case in which the eigenvalues are precisely the opposite of those for complete graphs.

\subsection*{Digraphs with spectrum $\{-(n-1), 1^{(n-1)}\}$ }
\label{subsec:minuskn}

The tightness case of Theorem \ref{thm:rhodelta} shows when $\la_1(X)$ is large in terms of vertex degrees. By trying to do the converse -- make $\la_1(X)$ small -- we arrive to digraphs $K_3'$ and $K_4'$ (see Figure \ref{fig:k3k4prime}) with $H$-spectra $\{-2, 1^{(2)}\}$ and $\{-3, 1^{(3)}\}$, respectively. It is worth mentioning that each of $K_3'$ and $K_4'$ has its $H$-spectrum that is the negative of the spectrum of their underlying complete graphs $K_3$ and $K_4$, respectively.  Naturally, we may ask if there are other digraphs on $n$ vertices whose $H$-spectrum is $\{-(n-1), 1^{(n-1)}\}$. Such digraphs would have a large negative eigenvalue and small positive eigenvalues and thus exhibit extreme spectral behaviour, opposite to the behaviour of undirected graphs, whose spectral radius always equals the largest (positive) eigenvalue. We answer this in the negative and show that $K_3'$ and $K_4'$ are the only non-trivial digraphs with this property.

\begin{figure}[ht!]
\centering
\begin{tikzpicture}
	\begin{pgfonlayer}{nodelayer}
		\node [style=new] (0) at (-4, 2.25) {};
		\node [style=new] (1) at (-5, 0.5) {};
		\node [style=new] (2) at (-3, 0.5) {};
		\node [style=new] (3) at (-1, 2.5) {};
		\node [style=new] (4) at (-1, 0.5) {};
		\node [style=new] (5) at (1, 2.5) {};
		\node [style=new] (6) at (1, 0.5) {};
		\node [style=n] (7) at (-4, -0.25) {$K_3'$};
		\node [style=n] (8) at (0, -0.25) {$K_4'$};
	\end{pgfonlayer}
	\begin{pgfonlayer}{edgelayer}
		\draw [directed] (1) to (0);
		\draw [directed] (0) to (2);
		\draw [directed, bend right=15] (2) to (1);
		\draw [directed, bend right=15] (3) to (5);
		\draw [directed, bend right=15] (4) to (6);
		\draw [directed] (4) to (3);
		\draw [directed] (3) to (6);
		\draw [directed] (6) to (5);
		\draw [directed] (5) to (4);
		\draw [directed, bend right=15] (1) to (2);
		\draw [directed, bend right=15] (6) to (4);
		\draw [directed, bend right=15] (5) to (3);
	\end{pgfonlayer}
\end{tikzpicture}
\caption{$K_3'$ and  $K_4'$. \label{fig:k3k4prime}}
\end{figure}
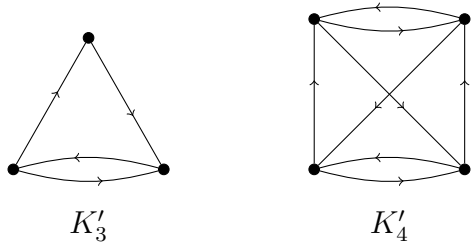

\begin{proposition}
\label{prop:hfam-minuskn}
If $X$ is a digraph such that $\sigma_H(X) = \{-(n-1), (-1)^{(n-1)}\}$, then $X \cong Y$ where $Y \in \{K_1, K_2, T_2, K_3', K_4'\}$, where $T_2$ is the oriented $K_2$.
\end{proposition}

\begin{proof}
Let $X$ be a digraph of order $n$ such that $H = H(X)$ has spectrum $\{-(n-1), (-1)^{(n-1)}\}$. If $n=1$, then $X \cong K_1$ and if $n=2$ then $X \cong K_2$ or $T_2$. Suppose now that $n\geq3$. The characteristic polynomial  of $H$ is $\phi(H, t) = (t + (n-1))(t-1)^{n-1} $.
Observe that $\phi(H(K_n)) = (t - (n-1))(t+1)^{n-1}$ and so its coefficients are
\[
[t^k] \phi(H, t) = \begin{cases} [t^k] \phi(H(K_n), t), & \text{if }n -k \text{ is even;}\\
-[t^k] \phi(H(K_n), t)  & \text{if }n -k \text{ is odd.}  \end{cases}
\]
In particular, $[t^{n-2}] \phi(H, t) = [t^{n-2}] \phi(H(K_n), t)$. Thus, by Corollary \ref{cor:cn-2}, $\Gamma(X)$ has the same number of edges as $K_n$, and so $\Gamma(X) \cong K_n$.

Also, $\tr(H^3) = -\tr(H(K_n)^3)$. By Proposition \ref{prop:traceh3}, we see that $\tr(H(K_n)^3) \allowbreak = 6\binom{n}{3}$. Therefore, $\tr(H^3) = - 6{n \choose 3} $ and so every $3$ vertices of $X$ must induce a digraph isomorphic to the digraph $K_3'$, the only possible triangle with negative weight.

Consider $G(X)$, the symmetric subgraph of $X$. If there is a path $uvw$ of length $2$ in $G(X)$, then $\{u,v,w\}$ induce a triangle of $X$ with more than one digon and hence not isomorphic to $K_3'$, which is a contradiction. Thus, each connected component of $G(X)$ is a copy of either $K_1$ or $K_2$. If $G(X)$ has $3$ or more components, then choosing three vertices from different components will give a triangle with no digons and hence not isomorphic to $K_3'$. Thus, $G(X)$ has at most $2$ components. Since $n\geq3$, we see that $G(X)$ has exactly two components and $n \in \{3,4\}$.
If $n = 3$, then $X \cong K_3'$ since $K_3'$ is an induced sub-digraph.

If $n=4$, then we may assume that $\{x_1,x_2\}$ and $\{x_3,x_4\}$ are digons in $X$. The deletion of any vertex of $X$ results in a digraph isomorphic to $K_3'$. When deleting $x_4$, we may assume that $x_1x_3$ and $x_3x_2\in E(X)$. Now it is easy to see that $x_2x_4\in E(X)$ (consider deleting $x_1$) and $x_4x_1\in E(X)$, giving us $K_4'$.
\end{proof}

\subsection*{Digraphs with a large negative $H$-eigenvalue}
\label{sec:hfam-largerho}

We define a digraph $X(a,b)$ on $2a + b$ vertices, where $a\geq 1$ and $b\geq 1$. The vertices of $X(a,b)$ consist of $X \cup Y \cup Z$, where $X = \{x_1, \ldots, x_a\}$, $Y = \{y_1, \ldots, y_a\}$ and $Z = \{z_1, \ldots, z_b\}$. The arcs are
\[\{x_j y_j, y_j x_j \mid j = 1,\ldots a\}\] and
\[\{x_j z_{\ell}, z_{\ell} y_j \mid j = 1,\ldots,a \text{ and } \ell = 1, \ldots b\}.\]
We see that $K_3'$ from the previous section is isomorphic to $X(1,1)$. Figure \ref{fig:largelambdan} shows $X(1,3)$ and $X(2,3)$.

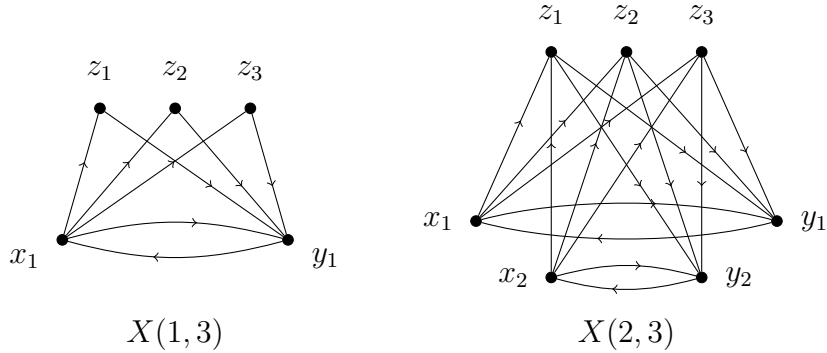
\begin{figure}[ht!]
\centering
\begin{tikzpicture}
	\begin{pgfonlayer}{nodelayer}
		\node [style=new] (0) at (-2, 1.25) {};
		\node [style=new] (1) at (-1, 1.25) {};
		\node [style=new] (2) at (0, 1.25) {};
		\node [style=new] (3) at (-2.5, -0.5) {};
		\node [style=new] (4) at (0.5, -0.5) {};
		\node [style=new] (5) at (4, -1) {};
		\node [style=new] (6) at (3, -0.25) {};
		\node [style=new] (7) at (7, -0.25) {};
		\node [style=new] (8) at (4, 2) {};
		\node [style=new] (9) at (5, 2) {};
		\node [style=new] (10) at (6, 2) {};
		\node [style=new] (11) at (6, -1) {};
		\node [style=n] (12) at (-1, -1.75) {$X(1,3)$};
		\node [style=n] (13) at (5, -1.75) {$X(2,3)$};
		\node [style=n] (14) at (-2, 1.75) {$z_1$};
		\node [style=n] (15) at (-1, 1.75) {$z_2$};
		\node [style=n] (16) at (0, 1.75) {$z_3$};
		\node [style=n] (17) at (4, 2.5) {$z_1$};
		\node [style=n] (18) at (5, 2.5) {$z_2$};
		\node [style=n] (19) at (6, 2.5) {$z_3$};
		\node [style=n] (20) at (-3, -0.75) {$x_1$};
		\node [style=n] (21) at (1, -0.75) {$y_1$};
		\node [style=n] (22) at (2.5, -0.25) {$x_1$};
		\node [style=n] (23) at (3.5, -1) {$x_2$};
		\node [style=n] (24) at (6.5, -1) {$y_2$};
		\node [style=n] (25) at (7.5, -0.25) {$y_1$};
	\end{pgfonlayer}
	\begin{pgfonlayer}{edgelayer}
		\draw [directed, bend left=15] (3) to (4);
		\draw [directed, bend left=15] (4) to (3);
		\draw [directed] (3) to (0);
		\draw [directed] (0) to (4);
		\draw [directed] (3) to (1);
		\draw [directed] (1) to (4);
		\draw [directed] (3) to (2);
		\draw [directed] (2) to (4);
		\draw [directed, bend left=15, looseness=0.75] (6) to (7);
		\draw [directed, bend left=15] (5) to (11);
		\draw [directed, bend left=15, looseness=0.75] (7) to (6);
		\draw [directed, bend left=15] (11) to (5);
		\draw [directed] (6) to (8);
		\draw [directed] (8) to (7);
		\draw [directed] (6) to (9);
		\draw [directed] (9) to (7);
		\draw [directed] (6) to (10);
		\draw [directed] (10) to (7);
		\draw [directed] (5) to (8);
		\draw [directed] (8) to (11);
		\draw [directed] (5) to (9);
		\draw [directed] (9) to (11);
		\draw [directed] (5) to (10);
		\draw [directed] (10) to (11);
	\end{pgfonlayer}
\end{tikzpicture}
\caption{Digraphs $X(1,3)$ and $X(2,3)$, constructed as examples of digraphs with a large negative $H$-eigenvalue. \label{fig:largelambdan}}
\end{figure}

\begin{lemma}
Digraph $X(a,b)$ defined above has $H$-spectrum \[\left\{ \frac{-1 + \sqrt{1+8ab}}{2},\ 1^{(a)},\ 0^{(b-1)},\ -1^{(a-1)},\ \frac{-1 - \sqrt{1+8ab}}{2} \right\}. \]
\end{lemma}

\begin{proof}
Let $H = H(X(a,b))$. We may write $H$ in the following form:
\[
H = \pmat{ \nils & I_a & i J_{a,b} \\ I_a & \nils & -i J_{a,b} \\ -i J_{b,a} & i J_{b,a} & \nils}
\]
where we recall that $I_n$ denotes the $n\times n$ identity matrix and $J_{m,n}$ denotes the $m\times n$ all-ones matrix.

Observe that the last $b$ rows are all identical and hence linearly dependent. This shows that $\rk(H) \leq 2a + b -(b-1)$, which implies that $H$ has $0$ as an eigenvalue with multiplicity at least $b-1$.

For $j = 1,\ldots, a$, let $\Zv_j = (\Ze_j \quad \Ze_j \quad \nils )^T$, where $\Ze_j$ is the $a$-dimensional $j$-th elementary vector. We see that $H \Zv_j = \Zv_j$ for $j = 1,\ldots, a$ and so $1$ is an eigenvalue of $H$ with multiplicity at least $a$.

Similarly, for $j = 1,\ldots, a-1$, let
\[
\Zw_j = \pmat{ \Ze_j - \Ze_a \\ -(\Ze_j - \Ze_a) \\ \nils }
\]
where $\Ze_n$ is defined as above. Then $H\Zw_j = - \Zw_j$ and so $-1$ is an eigenvalue of $H$ with multiplicity at least $a-1$.

We have found $2a + b -2$ eigenvalues of $H$. To find the remaining two eigenvalues, we will use the interlacing theorem. Partition the vertices of $X(a,b)$ (and consequently the rows and columns of $H$) into the sets $X$, $Y$ and $Z$. Each block of $H$, under this partition, has constant row sums and so this is an equitable partition of $H$. We obtain $B$, the quotient matrix corresponding to this partition as follows:
\[
B = \pmat{0 & 1 & ib \\ 1 & 0 & -ib \\ -ia & ia & 0}.
\]
We find that the characteristic polynomial of $B$ is
\[ \phi(B, t) = t^3 - (2ab + 1)t + 2ab = (t-1)(t^2 + t - 2ab)
\]
whose roots are $1$, $\tau = -\frac{1}{2} + \frac{1}{2}\sqrt{1 + 8ab}$ and
$\sigma = -\frac{1}{2} - \frac{1}{2}\sqrt{1 + 8ab}$.
The partition is equitable and so $\tau$ and $\sigma$ are also eigenvalues of $H$, by Corollary \ref{cor:interlacing-equitable}. Since for $a,b \geq 1$, $\tau$ and $\sigma$ are not equal to any of the eigenvalues of $H$ that we have already found. The trace formula gives that the last eigenvalue is another $1$. Thus, $H$ has spectrum $\{\tau, 1^{(a)}, 0^{(b-1)}, -1^{(a-1)}, \sigma\}$.
\end{proof}

In the sequel we will need a formula for eigenvalues of Cartesian products of digraphs.
Let $X$ and $Y$ be digraphs. The \emph{Cartesian product} of $X$ and $Y$, denoted by $X \Box Y$, is the graph with vertex set $V(X) \times V(Y)$ such that there is an arc from $(x_1, y_1)$ to $(x_2, y_2)$ when either $x_1x_2$ is an arc of $X$  and $y_1 = y_2$ or $y_1y_2$ is an arc of $Y$ and $x_1 = x_2$. The Hermitian adjacency matrix of $X \Box Y$ is
\[
H(X \Box Y) = H(X) \otimes I_{|V(Y)|} + I_{|V(X)|} \otimes H(Y)
\]
where $I_k$ is the $k\times k$ identity matrix. For definitions of Kronecker products of matrices and vectors, see \cite{CDS95}.

\begin{proposition}
\label{prop:hfam-prod}
If $X$ and $Y$ are digraphs with $H$-eigenvalues $\{\lambda_j\}_{j=1}^n$ and $\{\mu_k\}_{k=1}^m$ respectively, then $X\Box Y$ has $H$-eigenvalues $\lambda_j + \mu_k$ for $j = 1,\ldots, n$ and $k = 1,\ldots, m$.
\end{proposition}

\begin{proof}
Let $\Zv_1,\ldots, \Zv_n$ be an orthonormal eigenbasis of $H(X)$ such that $H(X)\Zv_j = \lambda_j \Zv_j$, for $j = 1,\ldots, n$. Let $\Zw_1,\ldots, \Zw_n$ be an orthonormal eigenbasis of $H(Y)$ such that $H(Y)\Zw_k = \mu_k \Zw_k$, for $k = 1,\ldots, m$. Observe that the vectors $\{\Zv_j \otimes \Zw_k \mid j \in \{1,\ldots, n\}, k \in \{1,\ldots, m\}\}$ form an orthonormal basis of $\cx^{nm}$. We see that
\[
\begin{split}
H(X \Box Y) \Zv_j \otimes \Zw_k &= \left(H(X) \otimes I_{m}\right)(\Zv_j \otimes \Zw_k) + \left(I_{n} \otimes H(Y)\right)( \Zv_j \otimes \Zw_k)\\
&= H(X)\Zv_j \otimes I_{m}\Zw_k + I_{n}\Zv_j \otimes H(Y)\Zw_k\\
&= \lambda_j\left(\Zv_j \otimes \Zw_k\right) + \mu_k \left(\Zv_j \otimes \Zw_k\right)\\
&= (\lambda_j + \mu_k)\left(\Zv_j \otimes \Zw_k\right)
\end{split} \]
for every $j \in \{1,\ldots, n\}$ and $k \in \{1,\ldots, m\}$.
\end{proof}

Note that $X(a,b)$ has $\rho(X(a,b)) - \lambda_1(X(a,b)) = 1$. We now use the Cartesian product to construct digraphs where this difference is much larger. We let $X^{\Box n}$ denote the $n$-fold Cartesian product of $X$ with itself; that is $X^{\Box n} = X \Box \cdots \Box X$, where there are $n$ terms in the product.

\begin{proposition}
\label{prop:3n}
The digraph $X_n = K_4'^{\,\Box n}$ has $\rho(X_n)=3n$ and $\la_1(X_n)=n$.
\end{proposition}

\begin{proof}
By applying Proposition \ref{prop:hfam-prod} $n$ times, the $H$-eigenvalues of $X_n$ are
\[ \Bigl\{ \sum_{j = 1}^n \beta_j \mid \beta_j \in \{-3, 1\} \Bigr\} = \{-3n, -3n+2, -3n+4, \dots, n-4,n-2,n\}.
\]
Thus $\rho(X_n) = 3n$ and $\lambda_1(X_n) = n$.
\end{proof}

The importance of the above examples is that they exhibit the extreme behavior as evidenced by our next result.

\begin{theorem}
\label{thm:max negative rho}
For every digraph $X$ we have
$$
    \la_1(X) \le \rho(X) \le 3\la_1(X).
$$
Both inequalities are tight.
\end{theorem}

\begin{proof}
The first inequality is clear by the definition of the spectral radius. Tightness is also clear by Proposition \ref{prop:3n}. To prove the second inequality, let $A$ be the 01-matrix corresponding to all digons in $X$, and let $L$ be the matrix corresponding to non-digons, so that $H=H(X)=A+L$.

It is easy to see that for every $\Zx\in \re^V$ ($V=V(X)$), we have $\Zx^TL\Zx=0$. Therefore, $\Zx^*H\Zx = \Zx^TA\Zx$.
By using the min-max formula (\ref{eq:minmax jth largest}) for the largest and smallest eigenvalues of $H$ and $L$ and noting that the same formula applies to the matrix $A$, in which case we need to consider only real vectors, we obtain the following inequality:
$$
   \la_1(A) = \max_{\Zx\in \re^V, \Vert \Zx\Vert=1} \Zx^TA\Zx = \max_{\Zx\in \re^V} \Zx^*H\Zx \le \max_{\Zz\in \cx^V, \Vert \Zz\Vert=1} \Zz^*H\Zz = \la_1(H).
$$
We are done if $\la_1(H) \ge \frac{1}{3}\rho(X)$. Thus, we may assume that $\la_1(A)\le\la_1(H)\le\frac{1}{3}\rho(X)$.
Now,
$$
   \rho(H)=\la_n(H)=|\min_{\Zz\in \cx^V, \Vert \Zz\Vert=1} \Zz^*H\Zz|.
$$
Suppose that the minimum (which is negative) is attained by a vector $\Zw\in \cx^V$ with $\Vert \Zw\Vert=1$.
Then
$$
  \rho(H) = |\Zw^*H\Zw| \le |\Zw^*L\Zw| + |\Zw^*A\Zw| \le |\Zw^*L\Zw| + \rho(A) \le |\Zw^*L\Zw| + \frac{1}{3}\rho(H).
$$
This implies that $|\Zw^*L\Zw|\ge \frac{2}{3}\rho(H)$. By \cite[Corollary 2.13]{LiuLi15}, see also Theorem \ref{thm:h-ori}, the spectrum of $L$ is symmetric about 0. Therefore there exists $\Zy\in \cx^V$ with $\Vert \Zy\Vert=1$ such that $\Zy^*L\Zy = |\Zw^*L\Zw|$.
Now,
$$
   \la_1(H) = \max_{\Zz\in \cx^V, \Vert \Zz\Vert=1} \Zz^*H\Zz \ge \Zy^*H\Zy = \Zy^*L\Zy + \Zy^*A\Zy \ge \frac{2}{3}\rho(H) - \rho(A) \ge
   \frac{1}{3}\rho(H).
$$
This completes the proof.
\end{proof}

It follows from the above proof that in the case of equality in the upper bound, the graph corresponding to the digons of $X$ is bipartite, its spectral radius is equal to $\frac{1}{3}\rho(X)$, and $\rho(L)=\frac{2}{3}\rho(X)$.

To end up this section we show that the spectral radius of $\Gamma(X)$ always majorizes $\rho(X)$.

\begin{theorem}
\label{thm:sp radius undirected majorizes}
For every digraph $X$ (with multiple edges allowed), $\rho(X)\le \rho(\Gamma(X))$.
\end{theorem}

\begin{proof}
Let $q=\pm\rho(X)$ be an eigenvalue of $X$, and let $\Zx$ be a unit eigenvector of $H(X)$ for its eigenvalue $q$. Define $\Zy$ by setting $\Zy(v)=|\Zx(v)|$, $v\in V(X)$. Then it is easy to see by using the triangular inequality that $\Zy^*A(\Gamma(X))\Zy \ge |\Zx^*H(X)\Zx| = \rho(X)$. Since $\Zy$ has norm 1, this implies that $\rho(\Gamma(X)) \ge \Zy^*A(\Gamma(X))\Zy \ge \rho(X)$, which we were to prove.
\end{proof}

\section{$H$-Eigenvalues symmetric about $0$}\label{sec:h-bip}

For a digraph $X$, its $A$-eigenvalues are symmetric about $0$ if and only if $\Gamma(X)$ is bipartite. We may also consider digraphs $X$ whose $H$-eigenvalues are symmetric about $0$. Note that the eigenvalues of $H$ are real and can be ordered as $\lambda_1 \geq \cdots \geq \lambda_n$, they are symmetric about $0$ if and only if $\lambda_j = - \lambda_{n-j+1}$ for $j = 1,\ldots, n$. The following proposition appears in \cite{LiuLi15}.

\begin{proposition}[\cite{LiuLi15}]
\label{prop:h-bip}
For a digraph $X$, if $\Gamma(X)$ is bipartite, then the $H$-eigenvalues are symmetric about $0$.
\end{proposition}

\begin{figure}[ht!]
\begin{center}
\begin{tikzpicture}
	\begin{pgfonlayer}{nodelayer}
		\node [style=new] (0) at (0, 1) {};
		\node [style=new] (1) at (-1, -0.5) {};
		\node [style=new] (2) at (1, -0.5) {};
	\end{pgfonlayer}
	\begin{pgfonlayer}{edgelayer}
		\draw [directed] (1) to (0);
		\draw [directed] (0) to (2);
		\draw [directed] (1) to (2);
	\end{pgfonlayer}
\end{tikzpicture}
\end{center}
\caption{Digraph $\widetilde{C_3}$ with $H$-eigenvalues symmetric about $0$ whose underlying graph is not bipartite. \label{fig:hsymnotbip}}
\end{figure}
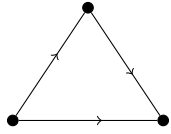

The converse to Proposition \ref{prop:h-bip} is not true. For example, the digraph $\widetilde{C_3}$ in Figure \ref{fig:hsymnotbip} has eigenvalues $\pm \sqrt{3}, 0$.
In fact, every oriented graph has $H$-eigenvalues symmetric about $0$. This was proved in \cite[Corollary 2.13]{LiuLi15}; see also \cite{CaCietal12}. Here we give a different and simpler proof.

\begin{theorem}[\cite{LiuLi15}]
\label{thm:h-ori}
If $X$ is an oriented graph, then the $H$-spectrum of $X$ is symmetric about $0$.
\end{theorem}

\begin{proof}
Let $X$ be an oriented graph on $n$ vertices and $H =H(X)$. Let the eigenvalues of $H$ be $\lambda_1 \geq \lambda_2 \geq \ldots \geq \lambda_n$. The matrix $iH$ is skew-symmetric with purely imaginary eigenvalues $i\lambda_1, \ldots, i\lambda_n$. Since $iH$ has entries $\pm 1$, its characteristic polynomial has real coefficients. Thus, every eigenvalue $\mu$ of $iH$ occurs with the same multiplicity as its complex conjugate. This implies that the spectrum of $H$ is symmetric about $0$.
\end{proof}

There are digraphs with $H$-eigenvalues symmetric about $0$, which are neither oriented nor have bipartite underlying graphs. Computationally, we verified that there are no such digraphs on fewer than $4$ vertices. For digraphs of order 4, we found out, using computer, that there are exactly seven $H$-cospectral classes with $H$-spectrum symmetric about $0$. They contain digraphs that are not oriented and their underlying graph needs not be bipartite. One of these classes contains exclusively such digraphs; this class contains $15$ non-isomorphic digraphs all of which have underlying graphs isomorphic to $K_4$, and each contains at least one digon. One graph from this class, $D$, is shown in Figure \ref{fig:symnotbipori}. The characteristic polynomial of $D$ is
\[
\phi(H(D), t) = t^4 - 6t^2 + 5.
\]

\begin{figure}[ht!]
\centering
\begin{tikzpicture}
	\begin{pgfonlayer}{nodelayer}
		\node [style=new] (0) at (-1, 1) {};
		\node [style=new] (1) at (1, 1) {};
		\node [style=new] (2) at (1, -1) {};
		\node [style=new] (3) at (-1, -1) {};
	\end{pgfonlayer}
	\begin{pgfonlayer}{edgelayer}
		\draw [directed] (0) to (1);
		\draw [directed] (1) to (2);
		\draw [directed, bend right=15, looseness=1.00] (2) to (3);
		\draw [directed] (0) to (3);
		\draw [directed] (1) to (3);
		\draw [directed] (2) to (0);
		\draw [directed, bend right=15, looseness=1.00] (3) to (2);
	\end{pgfonlayer}
\end{tikzpicture}
\caption{An example of a digraph on $4$ vertices, having $H$-eigenvalues symmetric about $0$, but not oriented and not bipartite.}  \label{fig:symnotbipori}
\end{figure}
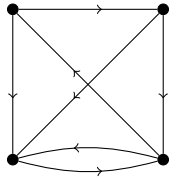

A common generalization of Proposition \ref{prop:h-bip} and Theorem \ref{thm:h-ori} follows from Theorem \ref{thm:hcp} using the fact that the spectrum is symmetric if and only if every even coefficient of the characteristic polynomial is zero. Thus, the following condition is certainly sufficient.

\begin{theorem}[\cite{LiuLi15}]
\label{thm:h-symmetric0}
If every odd cycle of $\Gamma(X)$ contains an even number of digons, then the $H$-spectrum of $X$ is symmetric about $0$.
\end{theorem}

The digraph in Figure \ref{fig:symnotbipori} shows that the above condition is only sufficient but not necessary.

A simple combinatorial characterization of digraphs with $H$-eigenvalues symmetric about $0$ is not known.

\section{$C^*$-algebra of a digraph}
\label{sec:h-diam}

The diameter of undirected graphs is bounded above by the number of distinct eigenvalues of its adjacency matrix. In this section we consider similar question for the Hermitian adjacency matrix.

Let $M$ be a Hermitian matrix (with at least one non-zero off-diagonal element to exclude trivialities). Let $\cM$ be the matrix algebra generated by $I, M, M^2, M^3, \ldots$ and let $\psi(M, t)$ be the minimal polynomial of $M$. Then, $\dim(\cM) = \deg(\psi(M, t)) - 1$. On the other hand, the degree of $\psi(M, t)$ is equal to the number of distinct eigenvalues of $M$.
If $M = A(X)$, the adjacency matrix of a digraph, it is easy to see that the dimension of $\cM$ is at least the diameter of the graph. This implies that the diameter of $X$ is smaller than the number of distinct eigenvalues of $A(X)$.

The case of the Hermitian adjacency matrix is quite different. For example, consider the modified directed cycle $\widetilde{C_n}$, obtained from a directed cycle by changing the orientation on one arc to the opposite direction.
Consider, in particular, the digraph $\widetilde{C_4}$ shown in Figure \ref{fig:c4tilde}.
We can compute that
\[
\phi(H(\widetilde{C_4}), t) = t^4 - 4t^2 + 4 = (t^2-2)^2
\]
and we see that $\widetilde{C_4}$ has exactly two distinct $H$-eigenvalues, but the diameter of $\Gamma(\widetilde{C_4})$ is~$2$.

\begin{figure}[ht!]
\centering
\begin{tikzpicture}
	\begin{pgfonlayer}{nodelayer}
		\node [style=new] (0) at (0, 1) {};
		\node [style=new] (1) at (0, -1) {};
		\node [style=new] (2) at (-1, 0) {};
		\node [style=new] (3) at (1, 0) {};
		\node [style=n] (4) at (0, 1.5) {$v_0$};
		\node [style=n] (5) at (1.5, 0) {$v_1$};
		\node [style=n] (6) at (0, -1.5) {$v_2$};
		\node [style=n] (7) at (-1.5, 0) {$v_3$};
	\end{pgfonlayer}
	\begin{pgfonlayer}{edgelayer}
		\draw [directed] (0) to (3);
		\draw [directed] (3) to (1);
		\draw [directed] (1) to (2);
		\draw [directed] (0) to (2);
	\end{pgfonlayer}
\end{tikzpicture}
\caption{Digraph $\widetilde{C_4}$, obtained from $C_4$ by reversing one arc. \label{fig:c4tilde}}
\end{figure}
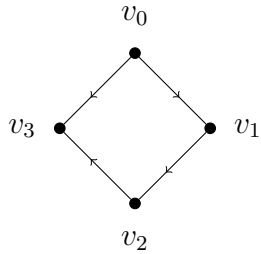

For $n\geq 3$, the $n$th \emph{necklace digraph,} denoted $N_n$, is an oriented graph on $3n$ vertices,
\[
V(N_n) = \{ v_j \mid j \in \ints_{2n} \} \cup \{ w_k \mid k \in \ints_n\}
\]
and with arcs
\[
\arcset(N_n) = \{v_jv_{j+1} \mid j \in \ints_{2n}\} \cup \{v_{2k}w_{k}, v_{2k + 2}w_{k} \mid k \in \ints_n\}.
\]
Let $C_j$ be the cycle $(v_{2j}, v_{2j + 1}, v_{2j + 2}, w_j, v_{2j})$ in $\Gamma(N_n)$. Each $v_{2j}$ lies on two of these cycles, $C_j$ and $C_{j-1}$. Every other vertex lies on a unique $C_j$. Figure \ref{fig:necklace4} shows $N_4$ with $C_0$ highlighted.

\begin{figure}[ht!]
\centering
\begin{tikzpicture}
	\begin{pgfonlayer}{nodelayer}
		\node [style=emph] (0) at (0, 3) {};
		\node [style=new] (1) at (0, -3) {};
		\node [style=new] (2) at (-3, 0) {};
		\node [style=emph] (3) at (3, 0) {};
		\node [style=emph] (4) at (2.25, 2.25) {};
		\node [style=emph] (5) at (1, 1) {};
		\node [style=new] (6) at (1, -1) {};
		\node [style=new] (7) at (-1, 1) {};
		\node [style=new] (8) at (-1, -1) {};
		\node [style=new] (9) at (2.25, -2.25) {};
		\node [style=new] (10) at (-2.25, -2.25) {};
		\node [style=new] (11) at (-2.25, 2.25) {};
		\node [style=n] (12) at (0, 3.5) {$v_0$};
		\node [style=n] (13) at (2.75, 2.5) {$v_1$};
		\node [style=n] (14) at (3.5, 0) {$v_2$};
		\node [style=n] (15) at (2.5, -2.75) {$v_3$};
		\node [style=n] (16) at (0, -3.5) {$v_4$};
		\node [style=n] (17) at (-2.75, -2.75) {$v_5$};
		\node [style=n] (18) at (-3.5, 0) {$v_6$};
		\node [style=n] (19) at (-2.75, 2.5) {$v_7$};
		\node [style=n] (20) at (-0.75, 0.75) {$w_3$};
		\node [style=n] (21) at (0.75, 0.75) {$w_0$};
		\node [style=n] (22) at (0.75, -0.75) {$w_1$};
		\node [style=n] (23) at (-0.75, -0.75) {$w_2$};
		\node [style=n] (24) at (1.5, 3) {\textcolor{magenta}{$C_0$}};
	\end{pgfonlayer}
	\begin{pgfonlayer}{edgelayer}
		\draw [directed, draw=magenta] (0) to (4);
		\draw [directed, draw=magenta] (4) to (3);
		\draw [directed] (3) to (9);
		\draw [directed] (9) to (1);
		\draw [directed] (1) to (10);
		\draw [directed] (10) to (2);
		\draw [directed] (2) to (11);
		\draw [directed] (11) to (0);
		\draw [directed,draw=magenta] (0) to (5);
		\draw [directed,draw=magenta] (3) to (5);
		\draw [directed] (3) to (6);
		\draw [directed] (1) to (6);
		\draw [directed] (1) to (8);
		\draw [directed] (2) to (8);
		\draw [directed] (2) to (7);
		\draw [directed] (0) to (7);
	\end{pgfonlayer}
\end{tikzpicture}
\caption{$N_4$ with $C_0$ in a lighter colour. \label{fig:necklace4}}
\end{figure}

To find the eigenvalues of $N_n$ we will use the following lemma.

\begin{lemma}
\label{lem:hfam-necklacecp}
For every $n\geq 3$ and $H = H(N_n)$, we have $H^3 = 4H$.
\end{lemma}

\begin{proof}
We will show that $H^3(u,v) = 4H(u,v)$. Observe that since the underlying graph $\Gamma := \Gamma(N_n)$ is a bipartite graph of girth $4$, we have that $H^3(u,u) = 0$ and $H^3(u,v) = 0$ if $\dist_{\Gamma}(u,v)$ is even or $\dist_{\Gamma}(u,v)>3$. We only need to consider the two cases where $\dist_{\Gamma}(u,v) \in \{1,3\}$.

Let us first consider the case when $u,v$ are adjacent in $\Gamma$. Since $H$ and $H^3$ are Hermitian, we need only check $u,v$ such that $uv \in \arcset(N_n)$. Let $\cW$ be the set of all walks of length $3$ from $u$ to $v$ in $\Gamma$. The following are all possible types of walks of length $3$ in $\Gamma$, starting at $u$ and ending at $v$:
\begin{enumerate}[(a)]
\item $W_1 = (u, v, u ,v)$;
\item $W_2 = (u, v, x ,v)$, where $x \in N_{\Gamma}(v) \setminus \{u\}$;
\item $W_3 = (u, w, u ,v)$, where $w \in N_{\Gamma}(u) \setminus \{v\}$; and
\item $W_4 = (u, w, x ,v)$, where $w \in N_{\Gamma}(u)\setminus \{v\}$, $x \in N_{\Gamma}(w)\setminus \{u\}$, and $xv \in \arcset(\Gamma)$.
\end{enumerate}

Every edge of $\Gamma$ which is traversed once in each direction in any $W_j$ contributes a factor of $1$ to the weight $\wt(W_j)$. Thus $\wt(W_1) = \wt(W_2) = \wt(W_3) = i$. To find the weight of $W_4$, we observe that every arc $uv$ lies on the unique cycle $C_j$, for some $j$, and that $W_4$ together with the arc $uv$ gives a $4$-cycle in $\Gamma$, which must then correspond to $C_j$. From this we immediately see that, for each such $uv$, there is exactly one walk from $u$ to $v$ isomorphic to $W_4$ and that $W_4$ is a path of length 3 on $C_j$. We may observe that all such paths either traverse two arcs in the backward direction and one in the forward direction, or all three arcs in the forward direction. In either case, $\wt(W_4) = -i$.

For every arc $uv$, one of $u$ or $v$ has degree $4$ in $\Gamma$ and the other has degree $2$. Then $\cW$ contains one walk isomorphic to $W_1$ and either $3$ walks isomorphic to $W_2$ and one isomorphic to $W_3$, or $3$ walks isomorphic $W_3$ and one isomorphic to $W_2$. Then, since $W_j$ for $j=1,2,3$ have the same weight, we get
\[ \begin{split}
H^3(u,v) &= \sum_{W \in \cW} \wt(W) \\
&= \wt(W_1) + \sum_{N_{\Gamma}(v) \setminus \{u\}} \wt(W_2) + \sum_{N_{\Gamma}(u) \setminus \{v\}} \wt(W_3) ~+~ \wt(W_4) \\
&= i + (3i + i)  - i = 4i = 4H(u,v)
\end{split}
\]
as claimed.

Finally, suppose tha $u,v$ are at distance 3 in $\Gamma$. In this case, since vertices lying on $C_j$ for any given $j$ can be at distance at most $2$, we have that $u \in C_j$ and $v \in C_{j \pm 1}$. Also note that either $u$ or $v$ is of degree $4$ in $\Gamma(X)$, thus equal to $v_{2k}$ for some $k \in \ints_n$.

Suppose that $u = v_{2k}$ for some $k$. Then $v \in C_{k + 1}$ or $v \in C_{k-2}$. In either case, there are two walks from $u$ to $v$ of opposite weight, and so $H^3(u,v) = 0 = H(u,v)$.
\end{proof}

\begin{corollary}
The $H$-spectrum of $N_n$ is $\sigma_H(N_n) = \{0^{(n)}, 2^{(n)}, -2^{(n)}\}$.
\end{corollary}

\begin{proof}
Let $H =H(N_n)$. From Lemma \ref{lem:hfam-necklacecp}, we see that the minimal polynomial of $H$ is $t^3 - 4t$. Since every eigenvalue of a matrix is a root of its minimal polynomial, the distinct eigenvalues of $H$ are $0$, $2$ and $-2$. Let $q,r$, and $s$ be the multiplicities of $0,2$, and $-2$, respectively. Since $\tr(H) = 0$, we see that $r = s$. By Proposition \ref{prop:traceh3}(ii), $\tr(H^2) = 2 |E(\Gamma(N_n))| = 8n$. Then $r(2^2 + (-2)^2) = 8n$, and so $r=n$. Since $q+2r = 3n$, we have that $q=n$.
\end{proof}

\begin{corollary}
There exists an infinite family of digraphs $\{X_j\}_{j=1}^{\infty}$ such that the diameter of $\Gamma(X_j)$ goes to $\infty$ as $j \rightarrow \infty$, while each $X_j$ has only three distinct $H$-eigenvalues.
\end{corollary}

\section{Cospectrality}
\label{sec:cospec}
\label{sec:hfam:forest}

In this section, we study properties of digraphs that are $H$-cospectral and describe some operations on digraphs that preserve the $H$-spectrum. In particular, we are motivated to consider digraph operations that preserve the $H$-spectrum and preserve the underlying graph.

The $H$-spectrum of a digraph $X$ does not determine if $X$ is strongly-connected, weakly-connected or disconnected. In Figure \ref{fig:connect}, we give an example of three digraphs, $X_1$, $X_2$ and $X_3$. A routine calculation shows that their characteristic polynomials are the same,
\[ \phi(H, t) = t^5 - 5t^3 + 2t^2 + 2t.
\]
Therefore, $X_1,X_2$, and $X_3$ are cospectral to each other.
Observe that $X_1$ is strongly connected, $X_2$ is weakly connected but not strongly connected, and $X_3$ is not even weakly connected.

\begin{figure}[ht!]
\begin{center}
\begin{tikzpicture}
		\node [style=new] (0) at (-4, 0) {};
		\node [style=new] (1) at (-5, 0) {};
		\node [style=new] (2) at (-4, -1) {};
		\node [style=new] (3) at (-3, 0) {};
		\node [style=new] (4) at (-4, 1) {};
		\node [style=new] (5) at (-1, 0) {};
		\node [style=new] (6) at (-2, 0) {};
		\node [style=new] (7) at (-1, -1) {};
		\node [style=new] (8) at (-1, 1) {};
		\node [style=new] (9) at (0, 0) {};
		\node [style=new] (10) at (1, 0) {};
		\node [style=new] (11) at (2, 0) {};
		\node [style=new] (12) at (2, 1) {};
		\node [style=new] (13) at (2, -1) {};
		\node [style=new] (14) at (3, 0) {};
		\node [style=n] (15) at (2, -1.5) {$X_3$};
		\node [style=n] (17) at (-1, -1.5) {$X_2$};
		\node [style=n] (18) at (-4, -1.5) {$X_1$};
		\draw [directed] (2) to (1);
		\draw [directed] (1) to (0);
		\draw [directed] (7) to (6);
		\draw [directed] (6) to (5);
		\draw [directed] (5) to (8);
		\draw [directed] (5) to (9);
		\draw [directed] (11) to (12);
		\draw [directed] (10) to (11);
		\draw [directed] (11) to (13);
		\draw [directed] (10) to (12);
		\draw [directed, bend left=45, looseness=0.75] (0) to (2);
		\draw [directed, bend left=45, looseness=0.75] (0) to (3);
		\draw [directed, bend left=45, looseness=0.75] (0) to (4);
		\draw [directed, bend right=315, looseness=0.75] (4) to (0);
		\draw [directed, bend left=45, looseness=0.75] (3) to (0);
		\draw [directed, bend left=45, looseness=0.75] (2) to (0);
		\draw [directed, bend right=315, looseness=0.75] (5) to (7);
		\draw [directed, bend left=45, looseness=0.75] (7) to (5);
		\draw [directed, bend left, looseness=0.75] (10) to (13);
		\draw [directed, bend left, looseness=0.75] (13) to (10);
\end{tikzpicture}
\end{center}
\caption{$H$-cospectral digraphs on 5 vertices with different connectivity properties. \label{fig:connect}}
\end{figure}
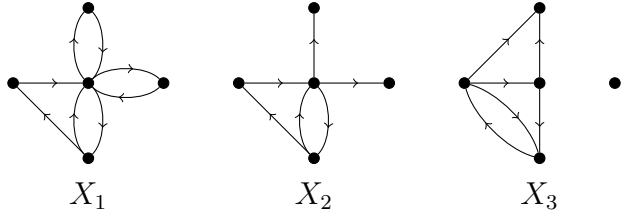

By the computation, as recorded in Tables \ref{tab:Hsmalldigrs} and \ref{tab:Asmalldigrs} in Section \ref{sec:hfam}, we see that, for small digraphs, the number of $H$-cospectral classes is smaller than the number of $A$-cospectral classes on the same number of vertices. (However, we expect that the opposite may hold when $n$ is large.)

By contrast, any two acyclic digraphs on the same number of vertices are $A$-cospectral (all their eigenvalues are 0); in particular, there are $A$-cospectral digraphs on $2$ vertices with non-isomorphic underlying graphs, whereas the smallest pair of $H$-cospectral digraphs with non-isomorphic underlying graphs have $4$ vertices (shown in Figure \ref{fig:cospecgammaneq}).

\begin{figure}[ht!]
\centering
\begin{tikzpicture}
	\begin{pgfonlayer}{nodelayer}
		\node [style=new] (0) at (-2, 0) {};
		\node [style=new] (1) at (-2, 1) {};
		\node [style=new] (2) at (-3, -0.75) {};
		\node [style=new] (3) at (-1, -0.75) {};
		\node [style=new] (4) at (3, -0.75) {};
		\node [style=new] (5) at (2, 1) {};
		\node [style=new] (6) at (1, -0.75) {};
		\node [style=new] (7) at (2, -0.13) {};
	\end{pgfonlayer}
	\begin{pgfonlayer}{edgelayer}
		\draw [directed] (0) to (1);
		\draw [directed] (0) to (2);
		\draw [directed] (0) to (3);
		\draw [directed] (6) to (5);
		\draw [directed] (5) to (4);
		\draw [directed] (4) to (6);
	\end{pgfonlayer}
\end{tikzpicture}
\caption{The smallest pair of $H$-cospectral digraphs with non-isomorphic underlying graphs. \label{fig:cospecgammaneq}}
\end{figure}
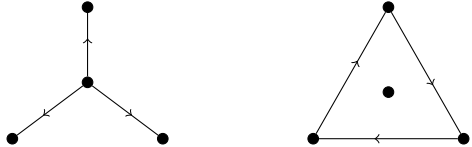

There are many cases of $H$-cospectral digraphs having isomorphic underlying graphs. In Section \ref{sec:hfam}, we will see that all orientations of an odd cycle are $H$-cospectral to each other and all digraphs whose underlying graph is an $n$-star are $H$-cospectral. From Table \ref{tab:h3s}, we see that every pair of $H$-cospectral digraphs on $3$ vertices have the same underlying digraph.

We will try to explain the spectral information about the underlying graph by looking at some $H$-spectrum preserving operations which do not change the underlying graph.

It is immediate from the definition of the Hermitian adjacency matrix that if $X$ is a digraph and $X^C$ is its converse, then $H(X^C) = H(X)^T = \comp{H(X)}$. This implies the following result.

\begin{proposition}
A digraph $X$ and its converse are $H$-cospectral.
\end{proposition}

Inspired by this, we now define a local operation on a digraph which will also preserve the spectrum with respect to the Hermitian adjacency matrix. For a digraph $X$ and a vertex $v$ of $X$, the \emph{local reversal of $X$ at $v$} is the operation of replacing every arc $xy$ incident with $v$ by its converse $yx$. We can extend this to the \emph{local reversal of $X$ at $S\subset V$} by taking the local reversal at $v$ for each $v\in S$. Observe that the order of reversals does not matter. If an arc $xy$ is incident to two vertices of $S$, then it is unchanged in the local reversal at $S$. We denote by $\delta(S)$ the arcs with exactly one end in $S$. Note that this operation generalizes the concept of switching-equivalence, defined earlier for tournaments.

\begin{proposition}
\label{prop:localreversal}
If $X$ is a digraph and $S \subset V(X)$ such that $\delta(S)$ contains no digons, then $X$ and the digraph obtained by the local reversal of $X$ at $S$ are H-cospectral.
\end{proposition}

\begin{proof}
Let $X'$ be the digraph obtained by the local reversal of $X$ at $S$. Let $M$ be the diagonal matrix indexed by the vertices of $X$ given by
\[
M_{uu} = \begin{cases} -1, & \text{if } u\in S; \\ +1, & \text{if } u\notin S. \end{cases}
\]
Consider $M^{-1}H(X)M$. Applying $M^{-1}=M$ on the left of $H(X)$ changes the sign for all rows indexed by vertices of $S$ and applying $M$ on the right changes the sign of all columns indexed by vertices of $S$. Since $\delta(S)$ contains no digons, this implies that $M^{-1}H(X)M = H(X')$. Consequently, the matrices $H(X)$ and $H(X')$ are similar and hence cospectral.
\end{proof}

Proposition \ref{prop:localreversal} cannot be generalized to the case when $\delta(S)$ contains digons. However, there is one exceptional case that is described next.

\begin{proposition}
\label{prop:cut of digons}
If $X$ is a digraph and $S \subset V(X)$ such that $\delta(S)$ contains only digons, then $X$ and the digraph obtained by replacing each digon $\{x,y\}$ $(x \notin S, y\in S)$ in the cut by the arc $xy$ are H-cospectral.
\end{proposition}

Proposition \ref{prop:cut of digons} follows directly from Theorem \ref{thm:hcp}, and the details are left to the reader.
Equivalently, the proof of Proposition \ref{prop:localreversal} works, where the only difference is that we take $M_{uu}=i$ if $u\in S$.

In a special case when each edge is a cutedge, we obtain the following corollary, which was proved earlier in \cite[Corollary 2.21]{LiuLi15}.

\begin{corollary}[\cite{LiuLi15}]
\label{cor:hfam-forests}
If $X$ is a digraph whose underlying graph is a forest, then $H(X)$ is cospectral with $A(\Gamma(X))$.
\end{corollary}

All of the above operations that preserve the $H$-spectrum can be generalized as discussed next. It all amounts to a simple similarity transformation that is based on the structure of Theorem \ref{thm:rhodelta}(i). Suppose that the vertex-set of $X$ is partitioned in four (possibly empty) sets, $V(X) = V_1\cup V_{-1}\cup V_i\cup V_{-i}$. An arc $xy$ or a digon $\{x,y\}$ is said to be of \emph{type} $(j,k)$ for $j,k\in\{\pm1, \pm i\}$ if $x\in V_j$ and $y\in V_k$. The partition is said to be \emph{admissible} if the following conditions hold:
\begin{enumerate}[(a)]
  \item There are no digons of types $(1,-1)$ or $(i,-i)$.
  \item All edges of types $(1,i),(i,-1),(-1,-i),(-i,1)$ are contained in digons.
\end{enumerate}
A \emph{four-way switching} with respect to a partition $V(X) = V_1\cup V_{-1}\cup V_i\cup V_{-i}$ is the operation of changing $X$ into the digraph $X'$ by making the following changes:
\begin{enumerate}[(a)]
  \item reversing the direction of all arcs of types $(1,-1),(-1,1),(i,-i),(-i,i)$;
  \item replacing each digon of type $(1,i)$ with a single arc directed from $V_1$ to $V_i$ and replacing each digon of type $(-1,-i)$ with a single arc directed from $V_{-1}$ to $V_{-i}$;
  \item replacing each digon of type $(1,-i)$ with a single arc directed from $V_{-i}$ to $V_1$ and replacing each digon of type $(-1,i)$ with a single arc directed from $V_i$ to $V_{-1}$;
  \item replacing each non-digon of type $(1,-i),(-1,i),(i,1)$ or $(-i,-1)$ with the digon.
\end{enumerate}

\begin{figure}
\centering
\includegraphics[scale=1.0]{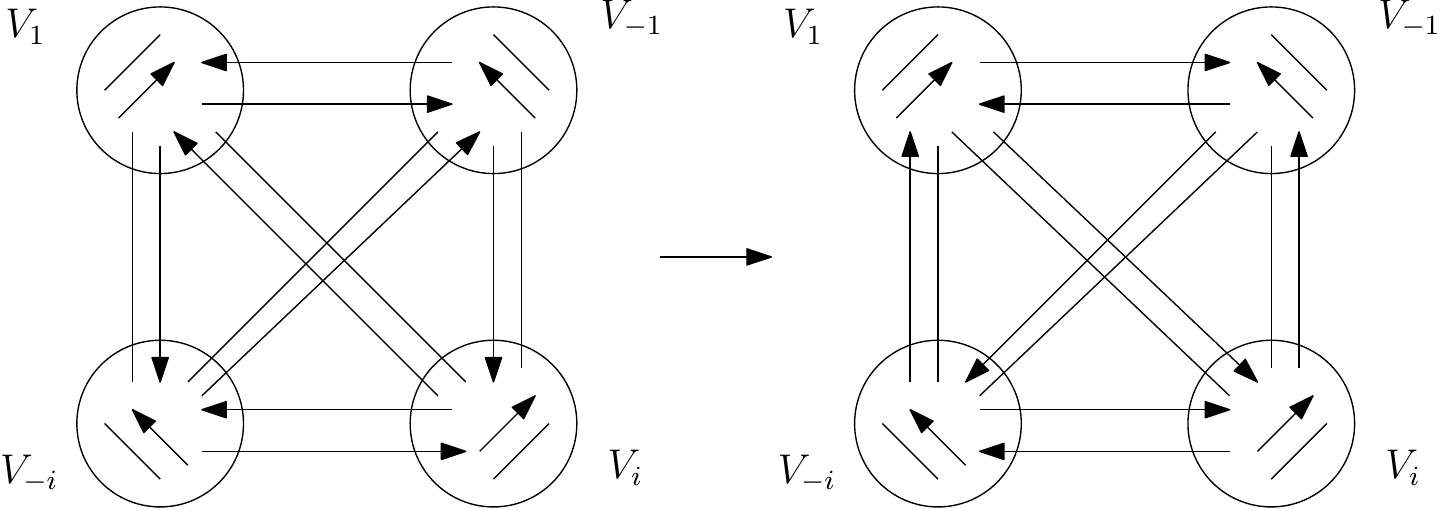}
\caption{Four-way switching on the admissible edges.}\label{fig:four-way-switching}
\end{figure}

\begin{theorem}
\label{thm:generalized switch}
If a partition $V(X) = V_1\cup V_{-1}\cup V_i\cup V_{-i}$ is admissible, then the digraph obtained from $X$ by the four-way switching is cospectral with $X$.
\end{theorem}

\begin{proof}
We use a similarity transformation with the diagonal matrix $S$ whose $(v,v)$-entry is equal to $j\in\{\pm1, \pm i\}$ if $v\in V_j$. The entries of the matrix $H' = S^{-1}HS$ are given by the formula
$$
    H'(u,v) = H(u,v) S(v)/S(u).
$$
It is clear that $H'$ is Hermitian and that its non-zero elements are in $\{\pm1, \pm i\}$. Note that the entries within the parts of the partition remain unchanged. Digons of type $(1,-1)$ would give rise to the entries $-1$, but since these digons are excluded for admissible partitions, this does not happen. On the other hand, any other entry in this part is multiplied by $-1$, and thus directed edges of types $(1,-1)$ or $(-1,1)$ just reverse their orientation. Similar conclusions are made for other types of edges. Admissibility is needed in order that $H'$ has no entries equal to $-1$. It turns out that $H'$ is the Hermitian adjacency matrix of $X'$. The details are easily read off from the following similarity using the diagonal $4\times 4$ matrix $S=\diag(1,-1,i,-i)$. The first one shows how digons are transformed:
$$
  S^{-1}\ \left[
             \begin{array}{cccc}
                1 & 1 & 1 & 1 \\
                1 & 1 & 1 & 1 \\
                1 & 1 & 1 & 1 \\
                1 & 1 & 1 & 1 \\
             \end{array}
          \right] \ S \ = \
          \left[
             \begin{array}{rrrr}
                1 & -1 & i & -i \\
                -1 & 1 & -i & i \\
                -i & i & 1 & -1 \\
                i & -i & -1 & 1 \\
             \end{array}
          \right].
$$
The second one shows the result for non-digon arcs:
$$
  S^{-1}\ \left[
             \begin{array}{cccc}
                i & i & i & i \\
                i & i & i & i \\
                i & i & i & i \\
                i & i & i & i \\
             \end{array}
          \right] \ S \ = \
          \left[
             \begin{array}{rrrr}
                i & -i & -1 & 1 \\
                -i & i & 1 & -1 \\
                1 & -1 & i & -i \\
                -1 & 1 & -i & i \\
             \end{array}
          \right].
$$
\end{proof}

Note the resemblance of Theorem \ref{thm:generalized switch} with Theorem \ref{thm:rhodelta}(i).

\subsection*{Digraphs that are $H$-cospectral with $K_n$}

In the undirected case, complete graphs have the property of being determined by their spectrum. In the case of digraphs with the Hermitian adjacency matrix it turns out that for each $n$, there are exactly $n$ non-isomorphic digraphs with the same $H$-spectrum as $K_n$. They can all be obtained by applying a single transformation of Proposition \ref{prop:cut of digons}.

For $a=1,\dots,n-1$, let $Y_{a,n-a}$ be the digraph of order $n$ that is obtained from the complete digraph $\digr(K_n)$ by replacing the digons between any of the first $a$ vertices (forming the set $A$) and any of the remaining $n-1$ vertices (forming the set $B$ by the single arc from $A$ to $B$ joining the same two vertices.
In other words, $Y_{a,n-a}$ consists of a copy of $\digr(K_a)$ and a copy of $\digr(K_{n-a})$, with all possible arcs from $\digr(K_a)$ to $\digr(K_{n-a})$. Figure \ref{fig:hfam-cospecwithkn} shows $Y_{2,3}$.

\begin{figure}[ht!]
\centering
\begin{tikzpicture}
	\begin{pgfonlayer}{nodelayer}
		\node [style=new] (1) at (-0.75, 1) {};
		\node [style=new] (2) at (0.75, 1) {};
		\node [style=new] (3) at (-1.75, -0.75) {};
		\node [style=new] (4) at (0, -0.75) {};
		\node [style=new] (5) at (1.75, -0.75) {};
		\node [style=n] (6) at (-1, 1.4) {$v_1$};
		\node [style=n] (7) at (1, 1.4) {$v_2$};
		\node [style=n] (8) at (-2.2, -1) {$w_1$};
		\node [style=n] (9) at (2.2, -1) {$w_3$};
		\node [style=n] (10) at (0, -1.15) {$w_2$};
	\end{pgfonlayer}
	\begin{pgfonlayer}{edgelayer}
		\draw [directed, bend left, looseness=1.00] (1) to (2);
		\draw [directed, bend left, looseness=1.00] (3) to (4);
		\draw [directed, bend left, looseness=1.00] (4) to (5);
		\draw [directed, bend left, looseness=1.00] (2) to (1);
		\draw [directed, bend left=60, looseness=1.25] (5) to (3);
		\draw [directed, bend right=60, looseness=0.75] (3) to (5);
		\draw [directed, bend left, looseness=1.00] (5) to (4);
		\draw [directed, bend left, looseness=1.00] (4) to (3);
		\draw [directed] (1) to (3);
		\draw [directed] (1) to (4);
		\draw [directed] (1) to (5);
		\draw [directed] (2) to (5);
		\draw [directed] (2) to (4);
		\draw [directed] (2) to (3);
	\end{pgfonlayer}
\end{tikzpicture}
\caption{Digraph $Y_{2,3}$ is $H$-cospectral to $K_{5}$.} \label{fig:hfam-cospecwithkn}
\end{figure}
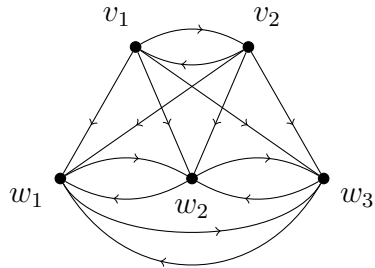

Proposition \ref{prop:cut of digons} implies that the digraph $Y_{a,n-a}$ as defined above is $H$-cospectral with $K_n$.

\begin{proposition}
\label{prop:hfam-cospeckn}
For each $n$, there are precisely $n$ non-isomorphic digraphs that have the same $H$-spectrum as $K_n$. These are the digraphs $K_n$ and $Y_{a,n-a}$, where $a = 1,\ldots, n-1$.
\end{proposition}

\begin{proof}
For $a \geq 1$ and $b=n-a$, we observe that the out-degree of a vertex of $Y_{a,b}$ is either $n-1$ or $b-1$. Then, for $c,d \geq 1$, the digraphs $Y_{a,b}$ and $Y_{c,d}$ have different sets of degrees unless $a=c$ and $b=d$. Thus, $Y_{a,b}$ is isomorphic to $Y_{c,d}$ if and only if $a=c$ and $b=d$. From Proposition \ref{prop:hfam-cospeckn}, we see that the digraphs
$Y_{a,n-a}$, $a\in \{1,\ldots n-1\}$,
are cospectral with $K_n$. Together with $K_n$ itself, there are at least $n$ non-isomorphic digraphs in the $H$-cospectral class containing $K_n$.

Let $X$ be a digraph which is $H$-cospectral with $K_n$. We will show that $X$ is isomorphic to one of $K_n$ and $Y_{a,b}$, where $a \in \{1,\ldots, n-1\}$ and $b = n-a$. Note that $\rho(X)=n-1$, thus by Theorem \ref{thm:rhodelta}, $\Gamma(X)=K_n$ and $X$ has the structure as in part (i) of that theorem. Since $\gamma(X)$ is complete, this can only be when at most two of the sets $V_j$ ($j\in\{\pm 1,\pm i\}$) are nonempty. Clearly, this gives either the digraph $\digr(K_n)$ or one of the digraphs $Y_{a,n-a}$.
\end{proof}

\subsection{Cospectral classes of digraphs whose underlying graph is a cycle}

At the end of this section, we determine the cospectral classes of all digraphs whose underlying graph is a cycle.
The directed cycle on $n$ vertices is denoted $D_n$ and $\widetilde{C}_n$ is the oriented cycle obtained from $D_n$ by reversing one arc. Further, we let $\widetilde{C}_n'$ be the digraph obtained from $D_n$ by replacing one arc by a digon, and $\widetilde{C}_n''$ be the digraph obtained from $D_n$ by replacing two consecutive arcs, the first one by reversing the arc, and the second one by adding the reverse arc, i.e. making it the digon.

We start with oriented cycles. We consider two digraphs $X$ and $Y$ to be \emph{equivalent under taking local reversals} if $Y$ can be obtained from $X$ by taking a series of local reversals.

\begin{lemma}
\label{lem:oddcycle}
All orientations of an odd cycle $C_{2m+1}$ are equivalent under taking local reversals. Each orientation of an even cycle $C_{2m}$ is equivalent either to $D_{2m}$ or to $\widetilde{C}_{2m}$ under taking local reversals.
\end{lemma}

\begin{proof}
Let $X$ be any orientation of $C_n$ on vertices $V = \{x_1,\ldots,x_n\}$ joined in the cyclic order.
For $j=2,3,\dots,n$ we consecutively make the local reversal at $x_j$ if the current graph obtained from $X$ has the arc $x_jx_{j-1}$, thus changing this arc to $x_{j-1}x_j$. This way we obtain either $D_n$ or $\widetilde{C}_n$. If $n$ is odd and we have $\widetilde{C}_n$, then we make additional reversals at $x_2,x_4,\dots, x_{n-1}$, thus transforming $\widetilde{C}_n$ into $D_n$. This completes the proof.
\end{proof}

Lemma \ref{lem:oddcycle} implies that every orientation of $C_n$ is $H$-cospectral with $D_{n}$ or with $\widetilde{C}_{n}$. We note that for even $n$, the spectra of $D_n$  and $\widetilde{C}_{n}$ are not equal (see Section~\ref{sec:hfam}).

\begin{proposition}
\label{prop:c_n}
Every digraph whose underlying graph is isomorphic to the $n$-cycle $C_n$ is $H$-cospectral to one of the following: $C_n$, $\widetilde{C}_n$, $\widetilde{C}_n'$ (when $n$ is even), and $\widetilde{C}_n''$ (when $n$ is odd).
\end{proposition}

\begin{proof}
Any two edges of $C_n$ form an edge-cut. By Proposition \ref{prop:cut of digons}, any edge-cut consisting of two digons can be changed to a directed cut without changing its $H$-spectrum, and any cut consisting of two arcs can be changed to its reverse by Proposition \ref{prop:localreversal}. Thus, we may assume that we have at most one digon. If there are no digons, Lemma \ref{lem:oddcycle} completes the proof. On the other hand, having one digon, we can make local reversals in the same way as in the proof of Lemma \ref{lem:oddcycle} to obtain either $\widetilde{C}_n'$ or $\widetilde{C}_n''$. If $n$ is even, then $\widetilde{C}_n''$ is cospectral to $C_n$, and if $n$ is odd, then $\widetilde{C}_n'$ is cospectral with $C_n$, so this gives the three cases of the proposition (details left to the reader).
\end{proof}

The $H$-eigenvalues of oriented cycles are treated fully in Section \ref{sec:hfam}.

\section{Digraphs with small spectral radius}
\label{sec:hevalspm1}

In the previous sections, we have seen that the $H$-eigenvalues of digraphs behave differently and somewhat strangely compared with the $A$-eigenvalues of graphs. It appears that graph invariants like diameter, minimum degree, and number of connected components cannot be bounded by the $H$-spectrum. However, since $H$ is Hermitian, we may use the interlacing theorem. Using interlacing, we can characterize all digraph with all $H$-eigenvalues small in absolute value. First we will look at a special case where all $H$-eigenvalues are equal to $1$ or $-1$, then we will consider the general case. Note that $mX$ where $X$ is a digraph denotes the union of $m$ disjoint copies of $X$.

\begin{theorem}
\label{thm:hevalpm1} A digraph $X$ has the property that $\lambda \in \{-1,1\}$ for each $H$-eigenvalue  $\lambda$ of $X$ if and only if $\,\Gamma(X) \cong m K_2$ for some $m$.
\end{theorem}

\begin{proof}
Let $X$ be a digraph on $n$ vertices having the property as in the statement of the theorem. Let $\lambda_1 \geq \cdots \geq \lambda_n$ be the eigenvalues of $H(X)$. Then $\lambda_i \in \{-1,1\}$ for $i = 1,\ldots, n$ by assumption. Since the $H$-eigenvalues of $X$ must sum to $\tr(H) = 0$, the multiplicity of $1$ and $-1$ are equal and $X$ has an even number vertices, say $n = 2m$. Observe that
\[\tr(H(X)^2) = \sum_{i=1}^n \lambda_i^2 = n = 2m.
\]
Lemma \ref{lem:Hdegs} gives that $\Gamma(X)$ has $m$ edges. If $X$ has an isolated vertex, then $X$ will have $0$ as an $H$-eigenvalue. Thus, every vertex must have degree at least $1$ in $\Gamma(X)$. Then $d_{\Gamma(X)}(v) = 1$ for every vertex $v$ and so $\Gamma(X) \cong m K_2$.
\end{proof}

\begin{theorem}
\label{thm:hevalinpm1}
For a digraph $X$ the following are equivalent:
\begin{enumerate}[\rm (a)]
\item $\sigma_H(X) \subseteq (-\sqrt{2}, \sqrt{2}\,)$;
\item $\sigma_H(X) \subseteq [-1, 1]$; and
\item every component of $X$ is either a single arc, a digon or an isolated vertex.
\end{enumerate}
\end{theorem}

\begin{proof}
Let $X$ be a digraph on $n$ vertices with $H$-eigenvalues $\lambda_1 \geq \cdots \geq \lambda_n$. Suppose that $\lambda_1 < \sqrt{2}$ and $\lambda_n > -\sqrt{2}$. Let $Y$ be an induced subdigraph of $X$ on $3$ vertices and let $\mu_1 \geq \mu_2 \geq \mu_3$ be the $H$-eigenvalues of $Y$. Applying the interlacing theorem, we obtain that
\[
-\sqrt{2} < \mu_i <\sqrt{2}
\]
for $i=1,2,3$. There are exactly $16$ digraphs on $3$ vertices and so we may compute their $H$-eigenvalues and determine which digraphs on $3$ vertices have all eigenvalues strictly between $-\sqrt{2}$ and $\sqrt{2}$. The digraphs on $3$ vertices grouped by $H$-cospectral classes are given in Table \ref{tab:h3s}. Following the naming of the digraphs in Table \ref{tab:h3s}, we see that $Y$ is isomorphic to one of $E_3$, $Z_5$ and $Z_6$. In other words, $Y$ is either the empty graph on $3$ vertices or a graph consisting of an isolated vertex and either one arc or one digon.

Since the choice of $Y$ was arbitrary, the above holds for every induced subdigraph of $X$ on $3$ vertices. Then, $\Gamma(X)$ does not have vertices of degree $2$ or more. Thus $\Delta(\Gamma(X)) \leq 1$ and so $\Gamma(X)$ consists of the union of disjoint copies of $K_2$ and isolated vertices. This shows that (a) implies (b).

The implications (b) $\Rightarrow$ (c) and (c) $\Rightarrow$ (a) are trivial, so this completes the proof.
\end{proof}

\begin{table}[ptbh]
\begin{center}
\begin{tabular}{|l|l|c|}
\hline
Polynomial $p(t)$ & Roots of $p(t)$ & All digraphs $X$ such that $\phi(H(X),t) = p(t)$\\
\hline
$t^3 - 2t$ & $\sqrt{2}, 0, -\sqrt{2} $ &
\begin{tikzpicture}[scale=0.65]
	\begin{pgfonlayer}{nodelayer}
		\node [style=new] (0) at (0, 1) {};
		\node [style=new] (1) at (0, -0) {};
		\node [style=new] (2) at (0, -1) {};
		\node [style=new] (3) at (1.5, 1) {};
		\node [style=new] (4) at (1.5, -0) {};
		\node [style=new] (5) at (1.5, -1) {};
		\node [style=new] (6) at (7.5, 1) {};
		\node [style=new] (7) at (7.5, -0) {};
		\node [style=new] (8) at (7.5, -1) {};
		\node [style=new] (9) at (4.5, -0) {};
		\node [style=new] (10) at (4.5, -1) {};
		\node [style=new] (11) at (4.5, 1) {};
		\node [style=new] (12) at (3, -0) {};
		\node [style=new] (13) at (3, 1) {};
		\node [style=new] (14) at (3, -1) {};
		\node [style=new] (15) at (6, -0) {};
		\node [style=new] (16) at (6, 1) {};
		\node [style=new] (17) at (6, -1) {};
		\node [style=n] (18) at (0, -1.5) {\scriptsize  $Z_1$};
		\node [style=n] (19) at (1.5, -1.5) {\scriptsize $\vec{P}_3$};
		\node [style=n] (20) at (3, -1.5) {\scriptsize $Z_2$};
		\node [style=n] (21) at (4.5, -1.5) {\scriptsize $Z_3$};
		\node [style=n] (22) at (6, -1.5) {\scriptsize $Z_4$};
		\node [style=n] (23) at (7.5, -1.5) {\scriptsize  $P_3$};
		\node [style=n] (24) at (3, 1.5) {};
	\end{pgfonlayer}
	\begin{pgfonlayer}{edgelayer}
		\draw [directed, bend right, looseness=1.00] (6) to (7);
		\draw [directed, bend right, looseness=1.00] (7) to (8);
		\draw [directed, bend right, looseness=1.00] (8) to (7);
		\draw [directed, bend right, looseness=1.00] (7) to (6);
		\draw [directed] (1) to (0);
		\draw [directed] (1) to (2);
		\draw [directed] (4) to (3);
		\draw [directed] (5) to (4);
		\draw [directed] (9) to (11);
		\draw [directed, bend right, looseness=1.00] (10) to (9);
		\draw [directed, bend right, looseness=1.00] (9) to (10);
		\draw [directed] (13) to (12);
		\draw [directed] (14) to (12);
		\draw [directed, bend right, looseness=1.00] (17) to (15);
		\draw [directed, bend right, looseness=1.00] (15) to (17);
		\draw [directed] (16) to (15);
	\end{pgfonlayer}
\end{tikzpicture}  \\
\hline
$t^3 - 3t + 2$ & $1,1,-2$ &
\begin{tikzpicture}[scale=0.65]
	\begin{pgfonlayer}{nodelayer}
		\node [style=new] (0) at (0, 1) {};
		\node [style=new] (1) at (-1, -0.5) {};
		\node [style=new] (2) at (1, -0.5) {};
		\node [style=n] (3) at (0, -1.25) { \scriptsize $K_3'$};
		\node [style=n] (24) at (0, 1.5) {};
	\end{pgfonlayer}
	\begin{pgfonlayer}{edgelayer}
		\draw [directed] (1) to (0);
		\draw [directed] (0) to (2);
		\draw [directed, bend left=15, looseness=1.00] (2) to (1);
		\draw [directed, bend left=15, looseness=1.00] (1) to (2);
	\end{pgfonlayer}
\end{tikzpicture} \\
\hline
$t^3 $ & $0,0,0$ &
\begin{tikzpicture}[scale=0.65]
	\begin{pgfonlayer}{nodelayer}
		\node [style=new] (0) at (0, 1) {};
		\node [style=new] (1) at (-1, 1) {};
		\node [style=new] (2) at (1, 1) {};
		\node [style=n] (3) at (0, 0.5) {\scriptsize $E_3$};
		\node [style=n] (24) at (0, 1.5) {};
	\end{pgfonlayer}
\end{tikzpicture}\\
\hline
$t^3 - 3t - 2$ & $2,-1,-1$ &
\begin{tikzpicture}[scale=0.65]
	\begin{pgfonlayer}{nodelayer}
		\node [style=new] (0) at (0, 1) {};
		\node [style=new] (1) at (-1, -0.5) {};
		\node [style=new] (2) at (1, -0.5) {};
		\node [style=n] (3) at (0, -1.25) { \scriptsize $K_3$};
		\node [style=n] (4) at (0, 1.5) {};
		\node [style=n] (5) at (3.5, -1.25) { \scriptsize $Y_{2,1}$};
		\node [style=new] (6) at (4.5, -0.5) {};
		\node [style=new] (7) at (2.5, -0.5) {};
		\node [style=new] (8) at (3.5, 1) {};
		\node [style=n] (9) at (7, -1.25) {\scriptsize  $Y_{1,2}$};
		\node [style=new] (10) at (8, -0.5) {};
		\node [style=new] (11) at (6, -0.5) {};
		\node [style=new] (12) at (7, 1) {};
	\end{pgfonlayer}
	\begin{pgfonlayer}{edgelayer}
		\draw [directed, bend left=15, looseness=1.00] (1) to (0);
		\draw [directed, bend left=15, looseness=1.00] (0) to (2);
		\draw [directed, bend left=15, looseness=1.00] (2) to (1);
		\draw [directed, bend left=15, looseness=1.00] (1) to (2);
		\draw [directed, bend left=15, looseness=1.00] (2) to (0);
		\draw [directed, bend left=15, looseness=1.00] (0) to (1);
		\draw [directed, bend left=15, looseness=1.00] (6) to (7);
		\draw [directed, bend left=15, looseness=1.00] (7) to (6);
		\draw [directed, bend left=15, looseness=1.00] (10) to (11);
		\draw [directed, bend left=15, looseness=1.00] (11) to (10);
		\draw [directed] (7) to (8);
		\draw [directed] (6) to (8);
		\draw [directed] (12) to (11);
		\draw [directed] (12) to (10);
	\end{pgfonlayer}
\end{tikzpicture} \\
\hline
$t^3 - t$ & $1, 0, -1$ &
\begin{tikzpicture}[scale=0.65]
	\begin{pgfonlayer}{nodelayer}
		\node [style=new] (0) at (0, 0.35) {};
		\node [style=new] (1) at (-1, -0.5) {};
		\node [style=new] (2) at (1, -0.5) {};
		\node [style=n] (3) at (0, -1.25) { \scriptsize $Z_5$};
		\node [style=n] (4) at (0, 0.65) {};
		\node [style=n] (5) at (3.5, -1.25) { \scriptsize $Z_6$};
		\node [style=new] (6) at (4.5, -0.5) {};
		\node [style=new] (7) at (2.5, -0.5) {};
		\node [style=new] (8) at (3.5, 0.35) {};
	\end{pgfonlayer}
	\begin{pgfonlayer}{edgelayer}
		\draw [directed, bend left=15, looseness=1.00] (2) to (1);
		\draw [directed] (6) to (7);
		\draw [directed, bend left=15, looseness=1.00] (1) to (2);
	\end{pgfonlayer}
\end{tikzpicture} \\
\hline
$t^3 - 3t$ & $\sqrt{3}, 0, -\sqrt{3}$ &
\begin{tikzpicture}[scale=0.65]
	\begin{pgfonlayer}{nodelayer}
		\node [style=new] (0) at (0, 1) {};
		\node [style=new] (1) at (-1, -0.5) {};
		\node [style=new] (2) at (1, -0.5) {};
		\node [style=n] (3) at (0, -1.25) { \scriptsize $Z_7$};
		\node [style=n] (4) at (0, 1.5) {};
		\node [style=n] (5) at (3.5, -1.25) {\scriptsize  $D_3$};
		\node [style=new] (6) at (4.5, -0.5) {};
		\node [style=new] (7) at (2.5, -0.5) {};
		\node [style=new] (8) at (3.5, 1) {};
		\node [style=new] (9) at (8, -0.5) {};
		\node [style=new] (10) at (6, -0.5) {};
		\node [style=n] (11) at (7, -1.25) { \scriptsize $\widetilde{C}_3$};
		\node [style=new] (12) at (7, 1) {};
	\end{pgfonlayer}
	\begin{pgfonlayer}{edgelayer}
		\draw [directed, bend left=15, looseness=1.00] (2) to (1);
		\draw [directed] (6) to (7);
		\draw [directed, bend left=15, looseness=1.00] (1) to (2);
		\draw [directed] (0) to (1);
		\draw [directed, bend left=15, looseness=1.00] (0) to (2);
		\draw [directed, bend left=15, looseness=1.00] (2) to (0);
		\draw [directed] (7) to (8);
		\draw [directed] (8) to (6);
		\draw [directed] (10) to (12);
		\draw [directed] (12) to (9);
		\draw [directed] (10) to (9);
	\end{pgfonlayer}
\end{tikzpicture} \\
\hline
\end{tabular}
\end{center}
\caption{$H$-cospectral classes of all digraphs on $3$ vertices. \label{tab:h3s} }
\end{table}

Note that Theorem \ref{thm:hevalinpm1} implies that $X$ has its $H$-spectrum equal to \[\{1^{(m)}, 0^{(k)}, -1^{(m)} \},\] where $k$ is the number of isolated vertices and $m$ is the number of components consisting of a single arc or a digon.

With some more work, we can characterize all digraphs with $H$-spectrum in $(-\sqrt{3}, \sqrt{3}\,)$. In order to do this, we need following corollaries of Proposition \ref{prop:c_n}, where we compute the spectra of all digraphs whose underlying graph is isomorphic to $C_4$ or $C_5$.
Using Sage, we found the characteristic polynomials and the $H$-eigenvalues of these digraphs. For $C_4$, they are collected in Table \ref{tab:c4s}. We can summarize this in the following corollary to Proposition \ref{prop:c_n}.

\begin{table}[ht]
\begin{center}
\begin{tabular}{|l|l|c|}
\hline
Digraph & Characteristic polynomial & Eigenvalues \\
\hline
$C_4$ & $t^4 - 4t^2$ & $\pm 2, 0^{(2)}$ \\
\hline
$\widetilde{C}_4$ & $t^4 - 4t^2 + 4$ & $\pm \sqrt{2}^{(2)}$ \\
\hline
$\widetilde{C}_4'$ & $t^4 - 4t^2 + 2$ & $\pm \sqrt{2 \pm \sqrt{2}}$ \\
\hline
\end{tabular}
\end{center}
\caption{$H$-eigenvalues of digraphs with $C_4$ as the underlying graph. \label{tab:c4s} }
\end{table}

\begin{corollary}
If the underlying graph of a digraph is isomorphic to the $4$-cycle $C_4$, then its $H$-spectrum is one of the following:
$\bigl\{\pm 2, 0^{(2)}\bigr\}$,
$\bigl\{\pm \sqrt{2}^{(2)} \bigr\}$, or
$\bigl\{\pm \sqrt{2 \pm \sqrt{2}}\,\bigr\}$.
\end{corollary}

We repeat the same process for digraphs whose underlying graph is $C_5$.
Using Sage, we found the $H$-eigenvalues of these graphs. They are given in Table \ref{tab:c5s}. We can summarize this in the following corollary to Proposition \ref{prop:c_n}.

\begin{table}[ht]
\begin{center}
\begin{tabular}{|l|l|c|}
\hline
Digraph & Characteristic polynomial & Eigenvalues \\
\hline
$C_5$ & $t^5 - 5t^3 + 5t -2$ & $2, \left(\frac{-1 \pm \sqrt{5}}{2} \right)^{(2)}$ \\
\hline
$D_5$ & $t^5 - 5t^3 + 5t$ & $0, \pm \sqrt{\frac{5 \pm \sqrt{5}}{2}}$ \\
\hline
$\widetilde{C}_5''$ & $t^5 - 5t^3 + 5t +2$ & $2, \left(\frac{1 \pm \sqrt{5}}{2} \right)^{(2)}$ \\
\hline
\end{tabular}
\end{center}
\caption{$H$-eigenvalues of digraphs with $C_5$ as the underlying graph.} \label{tab:c5s}
\end{table}

\begin{corollary}
If $Y$ is a digraph with $\Gamma(Y) \cong C_5$, then the $H$-spectrum of $Y$ is one of the following:
$\left\{2, \left(\frac{-1 \pm \sqrt{5}}{2} \right)^{(2)}\right\}$,
$\left\{0, \pm \sqrt{\frac{5 \pm \sqrt{5}}{2}} \right\}$, or
$\left\{2, \left(\frac{1 \pm \sqrt{5}}{2} \right)^{(2)}\right\}$.
\end{corollary}

We can now return to graphs with small spectral radius.

\begin{theorem}
\label{thm:hevalinsqrt3}
A digraph $X$ has $\sigma_H(X) \subseteq (-\sqrt{3}, \sqrt{3}\,)$ if and only if every component\/ $Y$ of $X$ has $\Gamma(Y)$ isomorphic to a path of length at most $3$ or to $C_4$, where in the latter case $Y$ is isomorphic to $\widetilde{C}_4$, or to one of the two strongly connected digraphs with two digons.
\end{theorem}

\begin{proof}
Let the $H$-eigenvalues of $X$ be $\lambda_1 \geq \cdots \geq \lambda_n$ and let $\Gamma = \Gamma(X)$. Let us assume that $\lambda_1 < \sqrt{3}$ and $\lambda_n > -\sqrt{3}$. Again, we consider $Y$, an induced subdigraph of $X$ on $3$ vertices and let $\mu_1 \geq \mu_2 \geq \mu_3$ be the $H$-eigenvalues of $Y$. Applying the interlacing theorem, we obtain that
\[
-\sqrt{3} < \mu_i <\sqrt{3}
\]
for $i=1,2,3$. Again, we consult Table \ref{tab:h3s} to see that $\Gamma(Y)$ is isomorphic to one of $P_3$, $K_2 \cup K_1$ or $E_3$. In other words, $\Gamma(Y)$ is acyclic.
Since the choice of $Y$ was arbitrary, the above holds for every induced subdigraph of $X$ on $3$ vertices. Then, $\Gamma$ does not contain a triangle.

If $\Gamma$ contains a vertex of degree at least $3$, then $\Gamma$ contains either an induced star on $4$ vertices or a triangle. We have already shown that there are no triangles, so $\Gamma$ must contain a star on $4$ vertices. It is easy to see that every digraph whose underlying graph is a $4$-star has $\sqrt{3}$ as an eigenvalue (see also Section \ref{sec:cospec}). By interlacing, this cannot happen in $X$. Thus, every vertex in $\Gamma$ has degree at most $2$.

The conclusion of the above is that the components of $\Gamma$ are paths and cycles. Suppose $W \subseteq V(X)$ induced a path of length $4$ in $\Gamma$. The spectral radius of $P_5$ is $\sqrt{3}$ and, since all digraphs with $P_5$ as its underlying graph are $H$-cospectral by Corollary \ref{cor:cutedge}, $W$ induces a subgraph of $X$ with maximum $H$-eigenvalue equal to $\sqrt{3}$, which is again impossible. Every cycle of length at least $6$ contains an induced path of length $4$ and every path of length at least $4$ contains an induced path of length $4$, so $\Gamma$ does not have any cycles of length greater than $5$ or paths of length greater than $3$.

We see from Table \ref{tab:c5s} that every digraph $Y$ such that $\Gamma(Y) \cong C_5$ has an eigenvalue strictly greater than $\sqrt{3}$. Also, we see from Table \ref{tab:h3s} that $\Gamma(Y)$ cannot be $C_3$. Thus we must have that each component of $X$ is either a path on at most four vertices or $\Gamma(Z) \cong C_4$. From Table \ref{tab:c4s}, we can see that $Z$ must be cospectral to $\widetilde{C}_4$. It is easy to see that there are three such digraphs. One is $\widetilde{C}_4$, the other two have two digons and they are strongly connected.

Conversely, a digraph all of whose components are as stated has spectral radius strictly smaller than $\sqrt{3}$. This completes the proof.
\end{proof}

It is interesting to consider for which values of $\alpha$ the number of weakly connected digraphs whose $H$-eigenvalues will all lie in the interval $(-\alpha, \alpha)$ or $[-\alpha, \alpha]$ will be finite. We see from Theorem \ref{thm:hevalinsqrt3} that there are only finitely many weakly connected digraphs with all $H$-eigenvalues in the interval $(-\sqrt{3}, \sqrt{3})$. It may be true that the same holds for every $\alpha$ with $0\leq \alpha <2$.  The directed paths show that there are infinitely many weakly connected digraphs whose $H$-eigenvalues will all lie in the interval $(-2, 2)$, thus $\alpha =2$ will not give the same conclusion.

\section{Examples}
\label{sec:hfam}

Computation on all isomorphism classes of digraphs of orders $2$, $3$, $4$, $5$, and $6$ were carried out using Sage open-source mathematical software system \cite{sage}, for both the adjacency matrix and the Hermitian adjacency matrix.  We include some data here to give an idea of how the Hermitian adjacency matrix behaves as compared to the adjacency matrix, on small digraphs. We refer to the set of all digraphs that are $H$-cospectral to a given digraph as an \emph{$H$-cospectral class.} A digraph is \emph{determined by its $H$-spectrum} if every digraph that is $H$-cospectral with it is also isomorphic to it.

\begin{table}[htdp]
\begin{center}
\begin{tabular}{|l|r|r|r|r|r|}
\hline
Order &	$2$ & $3$ & $4$ & $5$ & $6$ \\
\hline
Number of digraphs & $3$ & $16$ & $218$ & $9608$ & $1540944$ \\
\hline
Number of distinct characteristic polynomials  &$2$ & $6$ & $27$ &  $275$ & $10920$ \\
\hline
Number of $H$-cospectral classes such that: &&&&& \\
a) characteristic polynomial is irreducible over $\rats$ & $0$ & $0$ & $0$ & $0$ & $6 $\\
b) characteristic polynomial is square-free & $1$ & $3$ & $14$ & $214$ & $9980$ \\
\hline
Maximum size of a $H$-cospectral class	& $2 $ & $ 6 $ & $ 21 $ & $ 158 $ & $ 1338$ \\
\hline
Number of digraphs determined by $H$-spectrum &	$1 $ & $ 2	 $ & $ 3 $ & $ 5 $ & $ 16$ \\
\hline
Number of $H$-cospectral classes containing: &&&&& \\
a) no graphs & $0 $ & $  2 $ & $ 16$ & $ 242$ & $ 10769$ \\
b) only graphs & $1 $ & $ 1 $ & $ 1$ & $ 1 $ & $ 1$ \\
c) at least one graph and a digraph  &  $1 $ & $	3 $ & $	10 $ & $	32 $ & $150$ \\

\hline
\end{tabular}
\end{center}
\caption{The $H$-spectra of small digraphs.} \label{tab:Hsmalldigrs}
\end{table}

\begin{table}[htdp]
\begin{center}
\begin{tabular}{|l|r|r|r|r|r|}
\hline
Order &	$2 $ & $ 3 $ & $ 4 $ & $ 5 $ & $ 6$ \\
\hline
Number of digraphs & $3 $ & $ 16 $ & $ 218 $ & $ 9608 $ & $ 1540944$ \\
\hline
Number of distinct characteristic polynomials  & $2 $ & $ 7 $ & $ 46 $ & $  718 $ & $ 35239 $\\
\hline
Number of $A$-cospectral classes such that: &&&&& \\
a) characteristic polynomial is irreducible over $\rats$  & $0$ & $ 1 $ & $ 12$ & $277 $ & $19392 $\\
b) characteristic polynomial is square-free & $1$ & $5$ & $36$ &  $625$ & $33146$ \\
\hline
Maximum size of a $A$-cospectral class	& $2 $ & $ 6 $ & $ 42 $ & $ 592 $ & $ 15842$ \\
\hline
Number of digraphs determined by spectrum &	$1 $ & $ 5	 $ & $ 23 $ & $ 166  $ & $ 2317$  \\
\hline
Number of $A$-cospectral classes containing: &&&&& \\
a) no graphs & $0 $ & $  3 $ & $ 35 $ & $ 685 $ & $ 35088$ \\
b) only graphs & $1 $ & $ 2 $ & $ 5$ & $ 15$ & $ 69$ \\
c) at least one graph and a digraph  & $1 $ & $	2 $ & $	6 $ & $	18 $ & $82$ \\
\hline
\end{tabular}
\end{center}
\caption{The adjacency matrix spectra small digraphs. \label{tab:Asmalldigrs} }
\end{table}

Recall that $D_n$ denotes the directed cycle on $n$ vertices. The following is well-known.

\begin{lemma}
\label{lem:hdn-evs}
The $H$-eigenvalues of $D_n$ are
$2 \sin \frac{2\pi k}{n}$, $k = 0,\ldots, n - 1$.
\end{lemma}

Next we determine the eigenvalues of $\widetilde{C}_{n}$. Recall that this digraph is obtained from $D_n$ by reversing one arc and that it is cospectral with $D_n$ if $n$ is odd. In order to find the $H$-eigenvalues of $\widetilde{C}_n$, we will use known theorems about certain types of circulant matrices, which we will use again when discussing the $H$-eigenvalues of the transitive tournament.

Circulant matrices have been studied extensively, see \cite{HoJo13} for more information. A \emph{skew circulant matrix} is a circulant with a change in sign to all entries below the main diagonal. We will follow the notation of \cite{Da79} and define for $\Za = \pmat{a_0, \ldots, a_{n-1}}$ with real entries, the skew circulant matrix of $\Za$ as:
\[
\scirc(\Za) =
\pmat{a_0 & a_1 & \ldots & a_{n-1} \\
-a_{n-1} & a_0 & \ddots &  \vdots\\
\vdots & \ddots & \ddots &  a_1 \\
-a_1 & \ldots & -a_{n-1} & a_0}.
\]
The eigenvalues of a skew circulant matrix $\scirc(\Za)$ are easy to find, see, e.g. \cite[Section 3.2.1]{Da79}.

\begin{theorem}
\label{thm:skewevals}
The eigenvalues of $\scirc(\Za)$ are $\mu_j(\Za) = \sum_{k = 0}^{n-1} a_k \sigma^{(2j + 1)k},$ where $\sigma = e^{\frac{ \pi i}{n}}$, for $j = 0,1,\ldots, n-1$.
\end{theorem}

In particular, it will be useful to simplify this statement for skew-symmetric, skew circulant matrices.

\begin{corollary}
\label{cor:skewevals}
Suppose that $a_k = a_{n-k}$ for $k = 1,\ldots, n-1$. If $n$ is odd, then $\scirc(\Za)$  has eigenvalues
\[
\nu_j(\Za) = a_0 + 2i\, \sum_{k=1}^{\frac{n-1}{2}} a_k \sin \tfrac{k(2j + 1)\pi}{n}\, , \quad j = 0,1,\ldots, n-1.
\]
If $n$ is even, then $\scirc(\Za)$ has eigenvalues
\[
\nu_j(\Za) = a_0 + (-1)^j a_{\frac{n}{2}} i + 2i\, \sum_{k=1}^{\lfloor \frac{n-1}{2}\rfloor} a_k \sin \tfrac{k(2j + 1)\pi}{n}\, , \quad j = 0,1,\ldots, n-1.
\]
\end{corollary}

\begin{proof}
Let $\sigma = e^{\frac{ \pi i}{n}}$. If  $a_j = a_{n-j}$ for $j = 1,\ldots, n-1$, it is clear from the definition that $\scirc(\Za)$ is skew-symmetric. Observe that $\sigma^n = -1$. For any $j \in \{0, \ldots, n-1\}$ and $k \in \{1, \ldots, n-1\}$, consider the contribution of terms with $a_k$ and $a_{n-k}$ in $\mu_j(\Za)$ of Theorem \ref{thm:skewevals}:
\[
\begin{split}
a_k \sigma^{(2j + 1)k} + a_{n-k} \sigma^{(2j + 1)(n-k)}
&= a_k \left( \sigma^{(2j + 1)k} + \sigma^{(2j + 1)n} \sigma^{(2j + 1)(-k)} \right) \\
&= a_k \left( \sigma^{(2j + 1)k} - (\sigma^{-1})^{(2j + 1)k} \right) \\
&= a_k \left( 2 i \imag(\sigma^{(2j + 1)k} )  \right) \\
&= 2i a_k \sin \tfrac{(2j + 1)k\pi}{n}.
\end{split}
\]
If $n$ is odd, then we are done. If $n = 2m$, then we also have the contribution  $a_m\sigma^{(2j + 1)m} = a_m(-1)^j\, i$ that gives the result for the even case.
\end{proof}

We will use Corollary \ref{cor:skewevals} to find the eigenvalues of $\widetilde{C}_n$.

\begin{proposition}
\label{prop:hctilden-evs}
The eigenvalues of $\widetilde{C}_n$ $(n\ge3)$ are $2 \sin \frac{(2j + 1)\pi}{n}$, $j = 0,\ldots, n-1$.
\end{proposition}

\begin{proof}
Observe that $H(\widetilde{C}_n) = i\scirc(\Za)$ where $\Za = \pmat{a_0, \ldots, a_{n-1}}$ and $a_1 = a_{n-1} = 1$ and $a_j = 0$ for $j \notin \{1, n-1\}$. Then, by Corollary \ref{cor:skewevals}, the eigenvalues of $H(\widetilde{C}_n)$ are $i \nu_j(\Za)$ for $j = 0,\ldots, n-1$. Let $\sigma = e^{\frac{ \pi i}{n}}$ as before. Since $a_{\frac{n}{2}} = 0$, we obtain that $\nu_j(\Za) = 2i \sin \frac{(2j + 1)\pi}{n}$ for $j \in {0,\ldots, n-1}$.
Then, the eigenvalues of $H(\widetilde{C}_n)$ are $-2 \sin \frac{(2j + 1)\pi}{n}$ for $j = 0,\ldots, n-1$, which is easily seen to be the same as claimed.
\end{proof}

\subsection*{Transitive tournaments}\label{sec:hfam-transtmt}

The transitive tournaments are an important class of digraphs to study and the spectra of their skew-symmetric adjacency matrices have been studied as skew circulants and Toeplitz matrices in \cite{GrKiSh93, KSU03}.

Let $T_n$ denote the transitive tournament on $n$ vertices. See Figure \ref{fig:t5} which shows $T_5$.
Let $H_n =H(T_n)$ denote its Hermitian adjacency matrix.

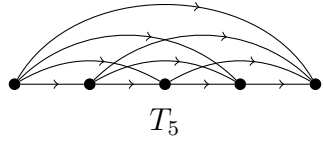
\begin{figure}[ht!]
\centering
\begin{tikzpicture}
	\begin{pgfonlayer}{nodelayer}
		\node [style=n] (0) at (0, -2) {$T_5$};
		\node [style=new] (1) at (-2, -1.5) {};
		\node [style=new] (2) at (0, -1.5) {};
		\node [style=new] (3) at (2, -1.5) {};
		\node [style=new] (4) at (1, -1.5) {};
		\node [style=new] (5) at (-1, -1.5) {};
	\end{pgfonlayer}
	\begin{pgfonlayer}{edgelayer}
		\draw [directed, bend left] (1) to (2);
		\draw [directed, bend left] (2) to (3);
		\draw [directed, bend left] (5) to (4);
		\draw [directed, bend left=45] (1) to (4);
		\draw [directed, bend left=45] (5) to (3);
		\draw [directed, bend left=60] (1) to (3);
		\draw [directed] (1) to (5);
		\draw [directed] (5) to (2);
		\draw [directed] (2) to (4);
		\draw [directed] (4) to (3);
	\end{pgfonlayer}
\end{tikzpicture}
\caption{Transitive tournament $T_5$ on $5$ vertices.  \label{fig:t5}}
\end{figure}

\begin{theorem}
The characteristic polynomial of the transitive tournament $T_n$ is
\[ \phi(T_n,t) =  \sum_{j=0}^{\lfloor \frac{n}{2} \rfloor } (-1)^j {n \choose 2j}  t^{n-2j}  = \frac{1}{2} (t+i)^{n}  +  \frac{1}{2} (t- i)^{n} .
\]
The $H$-eigenvalues of $T_n$ are
\[
    \sum_{k=1}^{\frac{n-1}{2}} 2 \sin \left( \frac{k(2j + 1)\pi}{n} \right)
\]
for $j = 0,\ldots, n-1$, when $n$ is odd, and
\[
    (-1)^{j}  + \sum_{k=1}^{\frac{n-2}{2}}  2 \sin \left( \frac{k(2j + 1)\pi}{n} \right)
\]
for $j = 0,\ldots, n-1$, when $n$ is even.
\end{theorem}

\begin{proof}
Observe that $H_n = i\scirc(\Za)$ where $\Za = \pmat{a_0, \ldots, a_{n-1}}$ and $a_1 = \ldots = a_{n-1} = 1$ and $a_0 = 0$. We see that $\scirc(\Za)$ is a skew-symmetric skew circulant and so we may apply Corollary \ref{cor:skewevals} to obtain that eigenvalues of $\scirc(\Za)$ are
\[
\nu_j(\Za) =
\begin{cases} a_0 + \sum_{k=1}^{\frac{n-1}{2}} a_k 2i \sin \left( \frac{k(2j + 1)\pi}{n} \right)  & \text{if } n \text{ is odd} \\
 a_0 + (-1)^j a_{\frac{n}{2}} i + \sum_{k=1}^{\lfloor \frac{n-1}{2}\rfloor} a_k 2i \sin \left( \frac{k(2j + 1)\pi}{n} \right) & \text{if } n \text{ is even}
\end{cases}
\]
for $j = 0,\ldots, n-1$. As the eigenvalues of $H_n$ are $i\nu_j$ for $j =0, \ldots, n-1$, we obtain the expressions for the $H$-eigenvalues as in the theorem.
\end{proof}

\bibliographystyle{plain}
\bibliography{spectra}

\end{document}